\documentclass[12pt,twoside]{amsart}
\usepackage{geometry}
\usepackage{amssymb,amsmath,amsthm, amscd, enumerate, mathrsfs}
\usepackage{graphicx, hhline}
\usepackage[all]{xy}
\usepackage[dvipdfmx]{hyperref}
\usepackage[shortlabels]{enumitem}
\setlist[enumerate]{leftmargin=56pt,labelsep=
8pt,itemsep=4pt,label=\upshape{(\thethm.\arabic*)}}
\usepackage{mathtools}
\mathtoolsset{showonlyrefs}
\usepackage{xcolor}

\newcommand{\RomI}{\uppercase\expandafter{\romannumeral 1}}

\title{On the Kodaira dimension}
\author{Osamu Fujino and Taro Fujisawa}
{\date{2024/7/9, version 0.13}
\subjclass[2020]{Primary 14D06; Secondary 14E30}
\keywords{logarithmic base change theorem, Kodaira dimensions, 
logarithmic Kodaira dimensions, 
variation of mixed Hodge structure, Higgs sheaves, 
pseudo-effectivity, weak positivity, semipositivity}
\address{Department of 
Mathematics, Graduate School of Science, 
Kyoto University, Kyoto 606-8502, Japan}
\email{fujino@math.kyoto-u.ac.jp}
\address{Department of Mathematics, School of Engineering, 
Tokyo Denki University, Tokyo, Japan}
\email{fujisawa@mail.dendai.ac.jp}

\DeclareMathOperator{\codim}{codim}
\DeclareMathOperator{\Supp}{Supp}

\DeclareMathOperator{\Gr}{Gr}

\DeclareMathOperator{\image}{Image}
\DeclareMathOperator{\red}{red}

\DeclareMathOperator{\coker}{Coker}

\DeclareMathOperator{\id}{id}
\DeclareMathOperator{\kos}{Kos}
\DeclareMathOperator{\res}{Res}

\newcommand{\gp}{^{\rm gp}}

\newtheorem{thm}{Theorem}[section]
\newtheorem{lem}[thm]{Lemma}
\newtheorem{cor}[thm]{Corollary}

\newtheorem{conj}[thm]{Conjecture}

\theoremstyle{definition}
\newtheorem{defn}[thm]{Definition}
\newtheorem{rem}[thm]{Remark}

\newtheorem*{ack}{Acknowledgments}  

\newtheorem{say}[thm]{}

\newtheorem{thm-s}{Theorem}[subsection]

\theoremstyle{definition}
\newtheorem{defn-s}[thm-s]{Definition}
\newtheorem{rem-s}[thm-s]{Remark}
\newtheorem{ex-s}[thm-s]{Example}
\newtheorem{say-s}[thm-s]{}

\makeatletter

\@addtoreset{equation}{section}
\makeatother
\begin{document}

\maketitle 

\begin{abstract}
We discuss the behavior of the Kodaira dimension 
under smooth morphisms. 
\end{abstract} 

\tableofcontents 

\section{Introduction}\label{p-sec1} 
We will discuss the behavior of the Kodaira dimension under 
smooth morphisms. 
Throughout this paper, we will work over $\mathbb C$, the field of 
complex numbers. 
One of the motivations of this paper is to understand \cite{park}. 
In \cite{park}, Sung Gi Park established the following striking and 
unexpected theorem. 

\begin{thm}[{Park's logarithmic base change theorem, 
see \cite[Theorem 1.2]{park}}]\label{p-thm1.1}
Let $X$, $Y$, and $Y'$ be smooth quasi-projective 
varieties and let $E$, $D$, and $D'$ be simple normal crossing 
divisors on $X$, $Y$, and $Y'$, respectively. 
Let $f\colon X\to Y$ and $g\colon Y'\to Y$ be projective 
surjective morphisms such that $f$ and $g$ are smooth 
over $Y\setminus D$, 
$f^{-1}(D)\subset E$ and $g^{-1}(D)\subset D'$, and that 
$E$ and $D'$ are relatively normal crossing over $Y\setminus D$. 
Let $X'$ be the union of the irreducible components 
of $X\times _Y Y'$ dominating $Y$ and $E':=g'^{-1}(E)\cup f'^{-1}(D')$. 
We consider the following commutative diagram: 
\begin{equation}
\xymatrix{
(X, E) \ar[d]_-f& (X', E') \ar[d]_-{f'}\ar[l]_-{g'}& (X'', E'')\ar[l]_-\mu
\ar[dl]^-{f''} \\ 
(Y, D) & (Y', D')\ar[l]^-g & 
}
\end{equation} 
where $\mu\colon X''\to X'$ is a projective 
resolution of singularities such that 
$\mu$ is an isomorphism over $Y\setminus D$, 
$E''$ is a simple normal crossing divisor on $X''$ such that 
$E''$ coincides with $E'$ over $Y\setminus D$ and that 
$(g\circ f'')^{-1}(D)\subset E''$. 
We put $\omega_{(X, E)}:=\omega_X\otimes \mathscr O_X(E)$, 
$\omega_{(Y, D)}:=\omega_Y\otimes \mathscr O_Y(D)$, 
$\omega_{(X, E)/(Y, D)}
:=\omega_{(X, E)}\otimes f^*\omega_{(Y, D)}^{\otimes -1}$, and so on.  
Then, for every positive integer $N$, 
there exists a generically isomorphic inclusion 
\begin{equation}\label{p-eq1.1}
\left(f_*\omega_{(X, E)/(Y, D)}^{\otimes N}\otimes 
g_*\omega_{(Y', D')/(Y, D)}^{\otimes N}\right)^{\vee\vee} 
\subset \left( h_*\omega_{(X'', E'')/(Y, D)}^{\otimes N}\right)
^{\vee\vee}  
\end{equation}
where $h:=g\circ f''$. 
\end{thm}

Precisely speaking, Park treated only the case where 
$f^{-1}(D)=E$ and $g^{-1}(D)=D'$ hold. 
However, we can easily see that 
\cite[Proposition 2.5]{park} implies 
the inclusion \eqref{p-eq1.1}. 
As a direct and easy consequence of Theorem \ref{p-thm1.1}, 
he obtained the following very important result. 

\begin{cor}[{Park's logarithmic fiber product trick, 
see \cite[Corollary 1.4]{park}}]\label{p-cor1.2}
Let $f\colon X\to Y$ be a projective surjective morphism 
of smooth quasi-projective varieties. Let $E$ and 
$D$ be simple normal crossing divisors on $X$ and $Y$, respectively. 
Assume that $f$ is smooth over $Y\setminus D$, $E$ is 
relatively normal crossing over $Y\setminus D$, and $f^{-1}(D)\subset E$. 
Let $s$ be a positive integer and let 
$f^s\colon X^s:=X\times _Y \cdots \times _YX\to Y$ be the $s$-fold fiber product 
of $f\colon X\to Y$. 
We put $E^s:=\sum _{i=1}^s p^*_i E$, where 
$p_i\colon X^s\to X$ is the $i$-th projection for every $i$. 
Let $\varphi\colon X^{(s)}\to X^s$ be a projective resolution 
of singularities of the union of the irreducible 
components of $X^s$ dominating $Y$ such that 
$\varphi$ is an isomorphism over $Y\setminus D$. 
Let $E^{(s)}$ be a simple normal crossing divisor 
on $X^{(s)}$ such that $E^{(s)}$ coincides 
with $E^s$ over $Y\setminus D$ and 
that $(f^{(s)})^{-1}(D)\subset E^{(s)}$, where 
$f^{(s)}:=f^s\circ \varphi$. 
Then we have the following generically isomorphic 
injection  
\begin{equation}\label{p-eq1.2}
\left( \bigotimes ^s f_*\left(\omega_{(X, E)/(Y, D)}^{\otimes N}
\right)\right)^{\vee\vee}
\hookrightarrow 
\left( f^{(s)}_* \left(\omega_{(X^{(s)}, E^{(s)})/(Y, D)}
^{\otimes N}\right)\right)^{\vee\vee}
\end{equation} 
for every positive integer $N$. 
Note that $X^{(s)}$ is smooth, $f^{(s)}$ is smooth over $Y\setminus D$, 
$E^{(s)}$ is a simple normal crossing 
divisor on $X^{(s)}$ and is relatively 
normal crossing over $Y\setminus D$, and $(f^{(s)})^{-1}(D)\subset E^{(s)}$.  
\end{cor}

We make a small remark on Corollary \ref{p-cor1.2}. 

\begin{rem}\label{p-rem1.3} 
If $f\colon X \to Y$ has connected fibers in Corollary \ref{p-cor1.2}, 
then $X^{(s)}$ is a smooth variety, that is, $X^{(s)}$ is connected, by 
construction. In general, $X^{(s)}$ may have some connected components. 
\end{rem}

In this paper, we will establish the following 
theorem, 
which is a slight generalization of 
\cite[Section 3]{park}, 
as an application of Corollary \ref{p-cor1.2} and 
the theory of variations of 
mixed Hodge structure. 
In \cite{park}, Theorem \ref{p-thm1.4} was 
treated under the assumption that 
$f^{-1}(D)=E$ holds. 

\begin{thm}[{see \cite[Section 3]{park}}]\label{p-thm1.4}
Let $f\colon X\to Y$ be a surjective morphism 
of smooth projective varieties and let $E$ and $D$ be simple 
normal crossing divisors on $X$ and $Y$, respectively. 
Assume that $f$ is smooth over $Y\setminus D$, 
$E$ is relatively normal crossing over $Y\setminus D$, 
and $f^{-1}(D)\subset E$. 
Let $\mathscr L$ be a line bundle on $Y$ such that 
there exists a nonzero homomorphism 
\begin{equation}
\mathscr L^{\otimes N} \to 
\left(f_*\omega_{(X, E)/(Y, D)}^{\otimes N}\right)^{\vee\vee}
\end{equation} 
for some positive integer $N$. 
Then there exists a pseudo-effective line bundle 
$\mathscr P$ on $Y$ and a nonzero homomorphism 
\begin{equation}
\mathscr L^{\otimes r} \otimes \mathscr P \to 
\left(\Omega^1_Y(\log D)\right)^{\otimes kr}
\end{equation} 
for some $r>0$ and $k\geq 0$. 
\end{thm}

By using Theorem \ref{p-thm1.4}, 
we will prove the following results. 

\begin{thm}[{see \cite[Theorem 1.5]{park}}]\label{p-thm1.5}
Let $f\colon X\to Y$ be a surjective morphism 
of smooth projective varieties with connected fibers. 
Let $E$ and $D$ be simple normal crossing divisors on 
$X$ and $Y$, respectively. 
Assume that $f^{-1}(D)\subset E$, $f$ is smooth over 
$Y\setminus D$, and $E$ is relatively normal crossing 
over $Y\setminus D$. 
We further assume that $\kappa (F, (K_X+E)|_F)\geq 0$ 
holds, where $F$ is a sufficiently general fiber of 
$f\colon X\to Y$. 
Then $\kappa (Y, K_Y+D)=\dim Y$ holds if and only if 
$\kappa (X, K_X+E)=\kappa (F, (K_X+E)|_F)+\dim Y$. 
\end{thm}

\begin{rem}\label{p-rem1.6}
In Theorem \ref{p-thm1.5}, it is well known that 
\begin{equation}
\kappa (X, K_X+E)=\kappa (F, (K_X+E)|_F)+\dim Y
\end{equation} 
holds 
under the assumption that $\kappa (Y, K_Y+D)=\dim Y$. 
This is due to Maehara (see \cite{maehara} and 
\cite{fujino-notes}). 
Hence the opposite implication is new and nontrivial. 
\end{rem}

Corollary \ref{p-cor1.7} is an obvious consequence of 
Theorem \ref{p-thm1.5}.
 
\begin{cor}\label{p-cor1.7} 
Let $f\colon X\to Y$, $D$, and $E$ be as in Theorem \ref{p-thm1.5}, 
that is, $f^{-1}(D)\subset E$, $f$ is smooth over $Y\setminus D$,  
and $E$ is relatively normal crossing over $Y\setminus D$. 
If $\kappa (X, K_X+E)=\dim X$, then 
$\kappa (Y, K_Y+D)=\dim Y$ and $\kappa (F, (K_X+E)|_F)=\dim F$ hold, 
where $F$ is a general fiber of $f\colon X\to Y$. 
\end{cor}

\begin{thm}[{see \cite[Theorem 1.7 (1)]{park}}]\label{p-thm1.8}
Let $f\colon X\to Y$ be a surjective morphism of 
smooth projective varieties and let $E$ and $D$ be 
simple normal crossing divisors on $X$ and $Y$, respectively. 
Assume that $f$ is smooth over $Y\setminus D$, 
$E$ is relatively normal crossing over $Y\setminus D$, 
and $f^{-1}(D)\subset E$. 
In this situation, if 
$\kappa (X, K_X+E-\varepsilon f^*D)\geq 0$ holds 
for some positive rational 
number $\varepsilon$, then there exists 
some positive rational number $\delta$ such that 
$K_Y+(1-\delta)D$ is pseudo-effective.   
\end{thm}

If $f^{-1}(D)=E$ in Theorems \ref{p-thm1.5} and 
\ref{p-thm1.8}, then they are nothing but \cite[Theorem 1.5]{park} and 
\cite[Theorem 1.7]{park}, respectively. 
In Theorem \ref{p-thm1.8}, if $\kappa (X, K_X+E)\geq 0$, then 
we can prove that $K_Y+D$ is pseudo-effective without 
using Theorem \ref{p-thm1.4}. 
It will be treated in Theorem \ref{p-thm4.5}. 

Let us consider a conjecture on 
the behavior of the (logarithmic) Kodaira dimension 
under smooth morphisms. 

\begin{conj}[{see \cite[Conjecture 3.6]{popa} and 
\cite[Conjecture 5.1]{park}}]\label{p-conj1.9}
Let $f\colon X\to Y$ be a surjective morphism of 
smooth projective varieties with connected fibers. 
Let $E$ and $D$ be simple normal crossing divisors on 
$X$ and $Y$, respectively. 
Assume that $f$ is smooth over $Y\setminus D$, 
$E$ is relatively normal crossing over $Y\setminus D$, 
and $f^{-1}(D)\subset E$. 
Then 
\begin{equation}
\kappa (X, K_X+E)=\kappa (Y, K_Y+D)+\kappa (F, (K_X+E)|_F)
\end{equation} 
holds, where $F$ is a sufficiently general fiber of $f\colon X\to Y$. 
\end{conj}

For the classical Iitaka 
subadditivity conjecture and some related topics, see \cite{fujino-iitaka}. 
For the details of relatively new conjectures on the behavior of 
the Kodaira dimension under morphisms 
of smooth complex varieties, see \cite{popa} and 
the references therein. 

We explain some related conjectures. 
Conjecture \ref{p-conj1.10} is a special 
case of the generalized abundance conjecture, which is 
one of the most important conjectures in the 
theory of minimal models. For the details 
of Conjecture \ref{p-conj1.10}, 
see \cite[Section 4.1]{fujino-foundations}. 

\begin{conj}[Generalized abundance conjecture for 
projective smooth pairs]\label{p-conj1.10}
Let $X$ be a smooth projective variety and let $E$ be 
a simple normal crossing divisor on $X$. 
Then 
\begin{equation}
\kappa (X, K_X+E)=\kappa_\sigma(X, K_X+E)
\end{equation} 
holds, where $\kappa _\sigma$ denotes 
Nakayama's numerical dimension.  
\end{conj}

Conjecture \ref{p-conj1.10} is still widely open. 
Conjecture \ref{p-conj1.10} contains 
Conjecture \ref{p-conj1.11} as a special case. 
For the details of the nonvanishing conjecture, 
see, for example, \cite[Section 4.8]{fujino-foundations} and 
\cite{hashizume-non}. 

\begin{conj}[Nonvanishing conjecture for projective smooth 
pairs]\label{p-conj1.11} 
Let $X$ be a smooth projective variety and let $E$ be 
a simple normal crossing divisor on $X$. 
Assume that $K_X+E$ is pseudo-effective. 
Then 
\begin{equation}
\kappa (X, K_X+E)\geq 0 
\end{equation} 
holds. 
\end{conj}

On Conjecture \ref{p-conj1.9}, we have a 
partial result, which is obviously a 
generalization of \cite[Theorem 1.12]{park}. 

\begin{thm}[{Superadditivity, 
see \cite[Theorem 1.12]{park}}]\label{p-thm1.12}
Let $f\colon X\to Y$ be a surjective morphism of 
smooth projective varieties with connected fibers. 
Let $E$ and $D$ be simple normal crossing divisors on 
$X$ and $Y$, respectively. 
Assume that $f$ is smooth over $Y\setminus D$, 
$E$ is relatively normal crossing over $Y\setminus D$, 
and $f^{-1}(D)\subset E$. 
We further assume that $\kappa (Y, K_Y+D)\geq 0$ and 
that the generalized abundance conjecture holds for sufficiently 
general fibers of the Iitaka fibration of $Y$ with respect 
to $K_Y+D$. 
Then 
\begin{equation}
\kappa (X, K_X+E)\leq \kappa (Y, K_Y+D)+\kappa (F, (K_X+E)|_F)
\end{equation} 
holds, where $F$ is a sufficiently general fiber of $f\colon X\to Y$. 
\end{thm}

We have already known that Conjecture \ref{p-conj1.10} holds 
true when $X\setminus E$ is affine. 
More generally, Conjecture \ref{p-conj1.10} 
holds true under the assumption that 
there exists a projective birational morphism 
$X\setminus E\to V$ onto an affine variety $V$. 
Hence we have: 

\begin{cor}\label{p-cor1.13}
Let $f\colon X\to Y$, $E$, and $D$ be as in Conjecture \ref{p-conj1.9}. 
We assume that $Y\setminus D$ is affine. 
Then the following superadditivity 
\begin{equation}\label{p-eq1.3}
\kappa (X, K_X+E)\leq \kappa (Y, K_Y+D)+\kappa (F, (K_X+E)|_F)
\end{equation} 
holds. 
\end{cor}

In Conjecture \ref{p-conj1.9}, 
we have already known that the subadditivity 
\begin{equation}
\kappa (X, K_X+E)\geq \kappa (Y, K_Y+D)+\kappa (F, (K_X+E)|_F)
\end{equation} 
follows from the generalized abundance conjecture 
(see Conjecture \ref{p-conj1.10}). 
For the details, see \cite{fujino-subadditivity}, 
\cite{fujino-corrigendum}, and \cite{hashizume}. 
Hence, if the generalized abundance conjecture 
holds true, then Conjecture \ref{p-conj1.9} 
is also true. Roughly speaking, we have: 

\begin{thm}\label{p-thm1.14}
Let $f\colon X\to Y$, $E$, and $D$ be as in Conjecture \ref{p-conj1.9}. 
We assume that Conjecture \ref{p-conj1.10} holds true. 
Then we have 
\begin{equation}
\kappa (X, K_X+E)=\kappa (Y, K_Y+D)+\kappa (F, (K_X+E)|_F), 
\end{equation} 
where $F$ is a sufficiently general fiber of $f\colon X\to Y$, 
that is, Conjecture \ref{p-conj1.9} holds true.  
\end{thm}

We look at the organization of this paper. 
In Section \ref{q-sec2}, we collect some basic definitions 
and results necessary for this paper for the reader's convenience. 
Subsection \ref{q-subsec2.1} collects some basic definitions. 
In Subsection \ref{q-subsec2.2}, we recall various notions of 
positivities. 
In Subsection \ref{q-subsec2.3}, we explain systems of 
Hodge bundles. 
In Section \ref{q-sec3}, we discuss Hodge theoretic weak 
positivity results. 
They seem to be more general than the usual 
ones slightly. 
In Section \ref{p-sec4}, 
we construct graded logarithmic Higgs sheaves and 
prove Theorem \ref{p-thm1.4} following \cite[Section 3.1]{park} 
closely. 
Here we use variations of mixed Hodge structure. 
In Section \ref{q-sec5}, we prove results 
explained in Section \ref{p-sec1}. 
In Section \ref{x-sec4}, we discuss variations 
of mixed Hodge structure necessary for Section \ref{p-sec4} for 
the sake of completeness. 
In the final section:~Section \ref{sec-fujisawa}, 
we prove a variant of semipositivity theorems
of Fujita--Zucker--Kawamata necessary for Section \ref{q-sec3}. 
Although we could not find any easily accessible accounts 
of the topics in Sections 6 and 7, the results may be more or less 
known to the experts. 

\begin{ack}\label{p-ack}
The first author was partially 
supported by JSPS KAKENHI Grant Numbers 
JP19H01787, JP20H00111, JP21H00974, JP21H04994, JP23K20787. 
The second author was partially
supported by JSPS KAKENHI Grant Number 
JP20K03542.
\end{ack}

\section{Preliminaries}\label{q-sec2}
In this section, we will collect some basic definitions and 
properties necessary for this paper. 

\subsection{Basic definitions}\label{q-subsec2.1} 
We will work over $\mathbb C$, the field of complex numbers. 
A {\em{variety}} means an irreducible and reduced 
separated scheme of finite type over $\mathbb C$. 

\begin{say-s}[$\kappa$ and $\kappa_\sigma$]\label{q-say2.1} 
Let $X$ be a smooth 
projective variety. 
Then $\kappa (X, \bullet)$ and $\kappa _\sigma(X, \bullet)$ denote 
{\em{the Iitaka dimension}} 
and {\em{Nakayama's numerical dimension}} of $\bullet$, respectively, 
where $\bullet$ is a $\mathbb Q$-Cartier divisor or a line bundle on $X$. 
For the details of $\kappa$ and $\kappa _\sigma$, see \cite{ueno}, 
\cite{mori}, \cite{nakayama}, and so on. 
\end{say-s}

\begin{say-s}[Canonical bundles and log canonical bundles]\label{q-say2.2}
Let $X$ be a smooth variety and let $E$ be a simple normal crossing 
divisor on $X$. 
Then we put 
\begin{equation}
\omega_X:=\det \Omega^1_X
\end{equation} 
and 
\begin{equation}
\omega_{(X, E)}:=\omega_X\otimes \mathscr O_X(E) =
:\omega_X(E). 
\end{equation} 
We note that 
\begin{equation}
\omega_{(X, E)}=\det \Omega^1_X(\log E). 
\end{equation}
Let $f\colon X\to Y$ be a surjective morphism 
of smooth varieties and let $D$ be a simple normal crossing 
divisor on $Y$. 
Then we put 
\begin{equation}
\omega_{(X, E)/(Y, D)}:=\omega_{(X, E)}\otimes 
f^*\omega^{\otimes -1}_{(Y, D)}= 
\omega_X(E)\otimes \left(f^*\omega_Y(D)\right)^{\otimes -1}. 
\end{equation}
Let $K_X$ be a Cartier divisor with $\mathscr O_X(K_X)\simeq 
\omega_X$. Then it is obvious that 
$\omega_{(X, E)}\simeq \mathscr O_X(K_X+E)$ holds. 
\end{say-s}

\begin{say-s}[Duals and double duals]\label{q-say2.3}
Let $\mathscr F$ be a coherent sheaf on a smooth 
variety $X$. 
Then we put 
\begin{equation}
\mathscr F^{\vee} :=\mathscr H\!om_{\mathscr O_X}(\mathscr F, \mathscr O_X)
\end{equation}
and 
\begin{equation}
\mathscr F^{\vee\vee} 
:=\mathscr H\!om_{\mathscr O_X}(\mathscr F^{\vee}, \mathscr O_X). 
\end{equation} 
We further assume that $\mathscr F$ is torsion-free. 
Then 
$\widehat \det \mathscr F$ and $\widehat S^\alpha(\mathscr F)$
denote $\left(\det \mathscr F\right) ^{\vee\vee}$ and 
$S^\alpha(\mathscr F)^{\vee\vee}$, respectively, where 
$S^\alpha(\mathscr F)$ is the $\alpha$-symmetric product of $\mathscr F$. 
\end{say-s}

\begin{say-s}[Sufficiently general fibers and 
general fibers]\label{q-say2.4} 
Let $f\colon X\to Y$ be a surjective morphism 
between varieties. 
Then a {\em{sufficiently general fiber}} (resp.~{\em{general fiber}}) 
$F$ of $f\colon X\to Y$ 
means that $F=f^{-1}(y)$, 
where $y$ is any closed point contained in a countable intersection 
of nonempty Zariski open sets (resp.~a nonempty Zariski open set) 
of $Y$. 
A sufficiently general fiber is sometimes called a {\em{very 
general fiber}} in the literature. 
\end{say-s}

\subsection{Weakly positive sheaves}\label{q-subsec2.2}
Let us recall the necessary definitions around various positivity. 
The following definition is well known and standard 
in the study of higher-dimensional algebraic varieties. 

\begin{defn-s}\label{q-def2.5}
Let $X$ be a projective variety and let $\mathscr L$ be a line bundle 
on $X$. 
\begin{itemize}
\item[(i)] $\mathscr L$ is {\em{big}} if 
$\mathscr L^{\otimes k}\simeq \mathscr H\otimes \mathscr O_X(B)$ for 
some positive integer $k$, 
an ample line bundle $\mathscr H$, and 
an effective Cartier divisor $B$ on $X$. 
\item[(ii)] $\mathscr L$ is {\em{nef}} if 
$\mathscr L\cdot C\geq 0$ holds 
for every curve $C$ on $X$. 
\item[(iii)] $\mathscr L$ is {\em{pseudo-effective}} 
if 
$\mathscr L^{\otimes m} \otimes \mathscr H$ is big 
for every positive integer $m$ and every ample 
line bundle $\mathscr H$ on $X$. 
\end{itemize} 
We note that a nef line bundle is always pseudo-effective. 
Let $\mathscr E$ be a locally free sheaf of finite rank on a projective 
variety $X$. 
\begin{itemize}
\item[(iv)] $\mathscr E$ is {\em{nef}} if 
$\mathscr O_{\mathbb P_X(\mathscr E)}(1)$ is a nef line bundle 
on $\mathbb P_X(\mathscr E)$.
\end{itemize}
We note that a nef locally free sheaf is sometimes 
called a {\em{semipositive}} locally free sheaf. 
\end{defn-s}

We recall the definition of weakly positive sheaves, 
which was first introduced by Viehweg. 
For the basic properties of weakly positive sheaves, 
see \cite{fujino-iitaka}. 

\begin{defn-s}[Weakly positive sheaves]\label{q-def2.6} 
Let $X$ be a normal quasi-projective variety and let $\mathscr A$ 
be a torsion-free coherent sheaf on $X$. 
We say that $\mathscr A$ is {\em{weakly positive}} 
if, for every positive integer $\alpha$ and every ample 
line bundle $\mathscr H$ on $X$, there exists a positive integer $\beta$ 
such that $\widehat S^{\alpha\beta}(\mathscr A) \otimes 
\mathscr H^{\otimes \beta}$ 
is generically generated by global sections, where 
$\widehat S^{\alpha\beta}(\mathscr A)$ denotes the 
double dual of the $\alpha\beta$-symmetric product of 
$\mathscr A$. 
\end{defn-s}

It is well known that $\widehat \det\mathscr A$ is weakly 
positive when $\mathscr A$ is a torsion-free weakly positive 
coherent sheaf on a normal quasi-projective variety. Moreover,  
a line bundle on a normal projective variety is 
pseudo-effective if and only if it is weakly positive. 
In general, the weak positivity does not behave well under extensions. 

\begin{rem-s}\label{q-rem2.7}
In \cite{ejiri-fujino-iwai}, we constructed a short exact sequence 
\begin{equation*}
0\to \mathscr E'\to \mathscr E\to \mathscr E''\to 0
\end{equation*} 
of locally free sheaves on a smooth projective surface such that 
$\mathscr E'$ and $\mathscr E''$ are pseudo-effective line bundles 
but 
$\mathscr E$ is not weakly positive. 
\end{rem-s} 

We make a small remark on \cite{ejiri-fujino-iwai} 
for the reader's convenience. 
Professor Robert Lazarsfeld pointed out that 
the following example answers \cite[Question 3.2]{ejiri-fujino-iwai} 
negatively. 

\begin{ex-s}[Gieseker]\label{q-ex2.8} 
There exists a rank two ample vector bundle 
$E$ on $\mathbb P^2$ sitting in an exact sequence 
\begin{equation*}
0\to \mathscr O_{\mathbb P^2}(-d)^{\oplus 2}
\to \mathscr O_{\mathbb P^2}(-1)^{\oplus 4}\to E\to 0, 
\end{equation*} 
where $d$ is a sufficiently large positive integer $d$ (see, 
for example, \cite[Example 6.3.67]{lazarsfeld}). Let $\pi\colon 
X\to \mathbb P^2$ be any generically finite surjective morphism 
from a smooth variety $X$. 
Then we see that $H^0(X, \pi^*E)=0$ since 
$H^1(X, \pi^*\mathscr O_{\mathbb P^2}(-d)^{\oplus 2})=0$ 
by the Kawamata--Viehweg vanishing theorem. 
In particular, $\pi^*E$ is not generically globally generated. 
\end{ex-s}

\subsection{Systems of Hodge bundles}\label{q-subsec2.3} 

Let $V_0=(\mathbb V_0, W_{0}, F_0)$ be a graded 
polarizable admissible variation of 
$\mathbb R$-mixed Hodge structure on a complex manifold $X_0$, 
where $\mathbb V_0$ is a local system of finite-dimensional 
$\mathbb R$-vector spaces on $X_0$, 
$W_{0}$ is an increasing filtration of $\mathbb V_{0}$ 
by local subsystems, and 
$F_{0}=\{F^p_0\}$ is the Hodge filtration. 
Then we obtain a Higgs bundle $(E_0, \theta_0)$ on $X_0$ by 
setting 
\begin{equation}
E_0=\Gr_{F_0}^\bullet \mathscr V_0=\bigoplus _p F^p_0/F^{p+1}_0
\end{equation} 
where $\mathscr V_0=\mathbb V_0\otimes \mathscr O_{X_0}$. 
Note that $\theta_0$ is induced by the Griffiths transversality 
\begin{equation}
\nabla\colon F^p_0\to \Omega^1_{X_0}\otimes _{\mathscr O_{X_0}}
F^{p-1}_0. 
\end{equation}
More precisely, $\nabla$ induces 
\begin{equation}
\theta^p_0\colon F^p_0/F^{p+1}_0\to \Omega^1_{X_0}\otimes 
_{\mathscr O_{X_0}}\left (F^{p-1}_0/F^p_0\right)
\end{equation} 
for every $p$. 
Then 
\begin{equation}
\theta_0=\bigoplus_p \theta^p_0\colon E_0\to \Omega^1_{X_0}\otimes 
_{\mathscr O_{X_0}}E_0.
\end{equation} 
The pair $(E_0, \theta_0)$ is usually called 
the {\em{system of Hodge bundles}} associated to $V_0=(\mathbb V_0, W_{0}, 
F_0)$ and $\theta_0$ is called the {\em{Higgs field}} of $(E_0, \theta_0)$. 

We further assume that $X_0$ is a Zariski open subset of 
a complex manifold $X$ such that 
$D=X\setminus X_0$ is a normal crossing divisor on $X$. 
We note that the local monodromy of $\mathbb V_0$ around $D$ is 
quasi-unipotent because $V_0$ is admissible. 
Let ${}^\ell \mathscr V$ be the lower canonical extension of 
$\mathscr V_0$ on $X$, 
that is, the {\em{Deligne extension}} of $\mathscr V_0$ on $X$ 
such that the eigenvalues of the residue of the connection 
are contained in $[0, 1)$. 
Let ${}^\ell F^p$ be the lower canonical extension of $F^p_0$, that is,  
\begin{equation}
{}^\ell F^p=j_*F^p_0\cap {}^\ell \mathscr V, 
\end{equation}  
where $j\colon X_0\hookrightarrow X$ is the natural open immersion, 
for every $p$. 
Since $V_{0}$ is admissible, 
${}^\ell F^p$ is a subbundle of ${}^\ell \mathscr V$ for 
every $p$, and 
we can extend $(E_0, \theta_0)$ to $(E, \theta)$ on $X$, 
where 
\begin{equation}
E=\Gr_F^\bullet {}^\ell \mathscr V =\bigoplus _p {}^\ell F^p/
{}^\ell F^{p+1}
\end{equation}  
and 
\begin{equation}
\theta\colon E\to \Omega^1_X(\log D)\otimes _{\mathscr O_X}E. 
\end{equation} 
When all the local monodromies of $V_0$ around $D$ are 
unipotent, we simply write $\mathscr V$ and $F^p$ to denote 
${}^\ell \mathscr V$ and ${}^\ell F^p$, respectively. 
We say that $\mathscr V$ (resp.~$F^p$) is the canonical extension of 
$\mathscr V_0$ (resp.~$F^p_0$). 

\begin{rem-s}\label{q-rem2.9} 
Although we formulated systems of Hodge bundles for graded 
polarizable {\em{admissible}} variations of $\mathbb R$-mixed Hodge structure,  
we do not need the relative monodromy weight filtration 
in this paper. We will only use Hodge bundles and their extensions. 
For the details of the necessary conditions, see \ref{item:pvhs} and 
\ref{item:extension} in \ref{7474} below and Remark \ref{7575}. 
\end{rem-s}

We will discuss a variant of semipositivity theorems 
of Fujita--Zucker--Kawamata in Section \ref{sec-fujisawa} below 
(see Theorem \ref{thm:1}). 

\section{Hodge theoretic weak positivity theorem}\label{q-sec3}
In this section, we will prove Theorem \ref{q-thm3.1}. 
For related topics, see \cite{fujita}, 
\cite{Zucker}, \cite{kawamata1}, \cite{zuo}, \cite{fujino-fujisawa1}, 
\cite{fujino-fujisawa-saito}, 
\cite{brunebarbe1}, 
\cite{popa-wu}, 
\cite{popa-schnell1}, 
\cite{fujisawa2}, 
\cite{fujino-fujisawa2},  
\cite{brunebarbe}, and so on. 
In this section, we will freely 
use the notation in Subsection \ref{q-subsec2.3}. 
Let us start with the following theorem:~Theorem \ref{thm-fujisawa}. 
As far as we know, there is no rigorous published proof 
of Theorem \ref{thm-fujisawa}. 
The reader can find an approach to Theorem \ref{thm-fujisawa} 
in Section \ref{sec-fujisawa} below, where 
we will give a Hodge theoretic semipositivity theorem which 
implies Theorem \ref{thm-fujisawa}. 
The approach in Section \ref{sec-fujisawa} is traditional and classical 
(see \cite{fujino-fujisawa1}, \cite{fujino-fujisawa2}, 
\cite{fujino-fujisawa-saito}, and 
\cite{fujisawa2}). 

\begin{thm}\label{thm-fujisawa}
Let $X$ be a smooth projective variety and 
let $X_0\subset X$ be 
a Zariski open subset such that $D=
X\setminus X_0$ is a simple normal crossing divisor on $X$. 
Let $V_0$ be a graded polarizable admissible 
variation of $\mathbb R$-mixed Hodge 
structure over $X_0$. 
We assume that all the local monodromies 
of $V_0$ around $D$ are unipotent. 
We consider the logarithmic Higgs field 
\begin{equation}
\theta\colon \Gr_F^\bullet\mathscr 
V\to \Omega^1_X(\log D)\otimes \Gr _F^\bullet \mathscr V.  
\end{equation} 
If $\mathscr E$ is a subbundle of $\Gr_F^\bullet\mathscr 
V$ and $\mathscr E$ is contained in the kernel of 
$\theta$, 
then $\mathscr E^{\vee}$ is a nef locally free sheaf on $Y$. 
\end{thm}

\begin{proof} 
We consider the dual variation of $\mathbb R$-mixed Hodge 
structure (see the description in 
\cite[2.5]{fujisawa2}). Then, by Theorem \ref{thm:1} below (see 
also \cite[Theorem 4.2]{fujisawa2}), 
we see that $\mathscr E^{\vee}$ is a nef locally free 
sheaf on $Y$. 
\end{proof} 

\begin{rem}\label{rem-new} 
When $V_0$ is pure in Theorem \ref{thm-fujisawa}, 
a slightly better result was proved in \cite[Theorem 1.5]{fujino-fujisawa2}. 
\end{rem}

The main result of this section is as follows. 

\begin{thm}[Hodge theoretic weak positivity theorem]\label{q-thm3.1}
Let $X$ be a smooth projective variety and 
let $X_0\subset X$ be 
a Zariski open subset such that $D=
X\setminus X_0$ is a simple normal crossing divisor on $X$. 
Let $V_0$ be a graded polarizable admissible 
variation of $\mathbb R$-mixed Hodge 
structure over $X_0$. 
If $\mathscr A$ is a coherent sheaf on $X$ 
such that $\mathscr A$ 
is contained in the kernel of the logarithmic Higgs field 
\begin{equation}
\theta\colon \Gr_F^\bullet {}^\ell\mathscr 
V\to \Omega^1_X(\log D)\otimes \Gr _F^\bullet {}^\ell\mathscr V, 
\end{equation} 
then the dual coherent sheaf $\mathscr A^{\vee}$ is weakly positive. 
\end{thm}

\begin{proof}[Proof of Theorem \ref{q-thm3.1}]
Let $U$ be the largest Zariski open subset of $X$ such that 
$\mathscr A|_U$ is locally free. 
Since $\mathscr A$ is torsion-free, we have $\codim _X(X\setminus U)\geq 2$. 
By Kawamata's unipotent reduction theorem, 
we have a finite surjective flat morphism 
$f\colon Y\to X$ from a smooth projective variety 
such that $f^{-1}D$ is a simple normal crossing divisor 
on $Y$, $f^{-1}V_0$ is a graded polarizable 
admissible variation of 
$\mathbb R$-mixed Hodge structure on 
$Y_0:=Y\setminus f^{-1}D$, and all the 
local momodromies of $f^{-1}V_0$ around $f^{-1}D$ 
are unipotent. 
By considering 
the canonical extension of the system of Hodge bundles 
associated to $f^{-1}V_0$, we have  
\begin{equation}
\theta_Y\colon \Gr_F^\bullet\mathscr 
V_Y\to \Omega^1_Y(\log f^{-1}D)\otimes \Gr _F^\bullet \mathscr V_Y,  
\end{equation} 
where $\mathscr V_Y$ is the canonical 
extension of $f^{-1}\mathbb V_0\otimes \mathscr O_{Y_0}$. 
Then $f^*\mathscr A$ is 
contained in the kernel of $\theta_Y$. 
If $(f^*\mathscr A)^{\vee}$ is weakly positive, then 
it is obvious that $f^*(\mathscr A^{\vee}|_U)$ is weakly positive. 
Hence we see that $\mathscr A^{\vee}$ is also 
weakly positive. Therefore, by replacing $\mathscr A$ and $V_0$ with 
$f^*\mathscr A$ and 
$f^{-1}V_0$, respectively, we may assume that 
$V_0$ has unipotent monodromies around $D$. 
We apply the flattening theorem to 
$\mathscr G/\mathscr A$, where $\mathscr G:=\Gr^{\bullet}_F\mathscr V$. 
Then we get a projective birational morphism 
$f\colon Y\to X$ from a smooth projective variety $Y$ such that 
$f^*(\mathscr G/\mathscr A)/\mathrm{torsion}$ is locally free 
and that $f^{-1}D$ is a simple normal crossing divisor on $Y$. 
By construction, we obtain a subbundle 
$\mathscr E$ of $f^*\mathscr G$ such that 
there exists a generically isomorphic injection 
$f^*\mathscr A\hookrightarrow 
\mathscr E$ on $f^{-1}(U)$. Thus $\mathscr E$ is contained in 
the kernel of $f^*\theta$, where 
\begin{equation}
f^*\theta\colon \Gr^{\bullet} _Ff^*\mathscr V
\to \Omega^1_Y(\log f^{-1}D)\otimes 
_{\mathscr O_Y}\Gr^{\bullet}_Ff^*\mathscr V. 
\end{equation} 
Hence $\mathscr E^{\vee}$ is a nef locally free sheaf on 
$Y$ by Theorem \ref{thm-fujisawa}. 
In particular, $\mathscr E^{\vee}$ is weakly positive. 
Since we have a generically isomorphic injection  
$\mathscr E^{\vee}\hookrightarrow (f^*\mathscr A)^{\vee}$ on $f^{-1}(U)$, 
$(f^*\mathscr A)^{\vee}$ is weakly positive on $f^{-1}(U)$. 
This implies that $\mathscr A^{\vee}$ is weakly positive 
on $U$. 
Hence, $\mathscr A^{\vee}$ is a weakly positive reflexive 
sheaf on $X$ since $\codim _X(X\setminus U)\geq 2$. 
This is what we wanted.  
\end{proof}

\section{Graded logarithmic Higgs sheaves}\label{p-sec4}

This section is a direct generalization of \cite[Section 3.1]{park}. 
The original idea goes back to Viehweg and Zuo (see \cite{viehweg-zuo1} 
and \cite{viehweg-zuo2}). 
Let us recall the definition of graded logarithmic 
Higgs sheaves following \cite{park} (see also \cite{viehweg-zuo1}, 
\cite{viehweg-zuo2}, \cite{popa-schnell1}, and so on). 

\begin{defn}[Graded logarithmic Higgs sheaves]\label{p-def4.1} 
Let $Y$ be a smooth variety and let $D$ be a simple 
normal crossing divisor on $Y$. 
A graded $\mathcal O_Y$-module $\mathscr F_\bullet =\bigoplus 
_k \mathscr F_k$ is a {\em{graded logarithmic 
Higgs sheaf}} with poles along $D$ if there exists a 
logarithmic Higgs structure 
\begin{equation}
\phi \colon \mathscr F_{\bullet}\to \mathscr F_{\bullet}
\otimes_{\mathscr O_{Y}} \Omega^1_Y(\log D)
\end{equation}
such that $\phi=\bigoplus _k \phi_k$, 
\begin{equation}
\phi_k \colon \mathscr F_k\to \mathscr F_{k+1}
\otimes_{\mathscr O_{Y}} \Omega^1_Y(\log D)
\end{equation} 
for every $k$, and 
\begin{equation}
\phi\wedge \phi\colon \mathscr F_{\bullet}\to
\mathscr F_{\bullet}\otimes _{\mathscr O_{Y}}\Omega^2_Y(\log D)
\end{equation} 
is zero. 
We put 
\begin{equation}
\mathscr K_k(\phi):=\ker\left(\phi_k\colon 
\mathscr F_{k}\to \mathscr F_{k+1}\otimes _{\mathscr O_{Y}}
\Omega^1_{Y}(\log D)\right), 
\end{equation} 
that is, $\mathscr K_{k}(\phi)$ is the kernel of the generalized 
Kodaira--Spencer map $\phi_k$ for every $k$. 
\end{defn}

Theorem \ref{p-thm4.2} is a slight generalization of \cite[Theorem 3.2]{park}. 
One of the motivations of this paper is to understand \cite[Theorem 3.2]{park}. 

\begin{thm}\label{p-thm4.2}
Let $f\colon X\to Y$ be a projective surjective 
morphism of smooth quasi-projective varieties. 
Let $D$ {\em{(}}resp.~$E${\em{)}} be 
a simple normal crossing divisor 
on $Y$ {\em{(}}resp.~$X${\em{)}}. 
Assume that $f$ is smooth over $Y\setminus D$, 
$E$ is relatively normal crossing over $Y\setminus D$, 
and $f^{-1}(D)\subset E$. 
Let $\mathscr L$ be a line bundle on $Y$ such that 
\begin{equation}
\kappa (X, \omega_{(X, E)/(Y, D)}\otimes 
f^*\mathscr L^{\otimes -1})=
\kappa (X, \omega_X(E)\otimes f^{*}(\omega_{Y}(D))^{\otimes -1}
\otimes f^{*}\mathscr L^{\otimes -1})\geq 0. 
\end{equation}  
Then there exists a graded logarithmic 
Higgs sheaf $\mathscr F_{\bullet}$ with poles along $D$ satisfying 
the following properties: 
\begin{itemize}
\item[{\em{(i)}}] $\mathscr L\subset \mathscr F_{0}$ and 
$\mathscr F_{k}=0$ for every $k<0$.
\item[{\em{(ii)}}] There exists a positive integer $d$ such that 
$\mathscr F_{k}=0$ for every $k>d$. 
\item[{\em{(iii)}}] $\mathscr F_{k}$ is a reflexive 
coherent sheaf on $Y$ for every $k$. 
\item[{\em{(iv)}}] 
Let $\mathscr A$ be a coherent subsheaf 
of $\mathscr F_\bullet$ contained in the kernel of $\phi$. 
Then $\mathscr A^{\vee}$ is weakly positive. 
In particular, 
$\mathscr K_k(\phi)^{\vee}$ is weakly positive. 
\end{itemize}
\end{thm}

Although the proof of Theorem \ref{p-thm4.2} is 
essentially the same as that of \cite[Theorem 3.2]{park}, 
we explain it in detail for the sake of completeness. 

\begin{proof}[Proof of Theorem \ref{p-thm4.2}]
We will closely follow the argument in \cite{park}. 
Since it is sufficient to construct 
$\mathscr F_{\bullet}$ on the complement 
of a codimension two closed 
subset in $Y$, 
we will freely remove suitable codimension 
two closed subsets from $Y$ throughout 
this proof. We put 
\begin{equation}
\mathscr N:=\omega_X(E)\otimes f^{*}\left(\omega_{Y}(D)\right)^{\otimes -1}
\otimes f^{*}\mathscr L^{\otimes -1}. 
\end{equation} 
Since $\kappa (X, \mathscr N)\geq 0$ by assumption, 
we can take a section $s$ of $\mathscr N^{\otimes m}$ for 
some positive integer $m$. 
Let $X'\to X$ be the cyclic cover of $X$ associated 
to $s$ and let $Z\to X'$ be a suitable 
resolution of singularities. We put $\psi\colon Z\to X$ and $h:=f\circ \psi\colon Z\to Y$. 
Hence we have the following commutative 
diagram:  
\begin{equation}
\xymatrix{
&Z\ar[dl]\ar[d]_-\psi\ar[dr]^-h& \\ 
X' \ar[r]& X\ar[r]_-f & Y. 
}
\end{equation}
Then there exists a natural inclusion $\psi^{*}\mathscr N^{\otimes -1}
\hookrightarrow \mathscr O_{Z}$. 
Let $E_Z$ be a simple 
normal crossing divisor on $Z$ such that $\psi^{-1}(E)\subset E_Z$ 
and that $\psi^{-1}(E)=E_Z$ holds over the generic point of $Y$. 
After removing a suitable codimension two 
closed subset from $Y$ and 
taking a birational modification of $Z$ suitably, 
we may further assume that 
there exists a smooth divisor $D'$ on $Y$ such that 
$h\colon Z\to Y$ is smooth over $Y\setminus D'$, 
$E_Z$ is a relatively normal crossing over $Y\setminus D'$, 
and $h^{-1}(D')\subset 
E_Z$. As usual, we put 
\begin{equation}
\Omega^{1}_{X/Y}(\log E):=
\coker\left(f^{*}\Omega^{1}_{Y}(\log D)\to \Omega^{1}_{X}(\log E)\right)
\end{equation}
and 
\begin{equation}
\Omega^{1}_{Z/Y}(\log E_{Z}):=
\coker\left(h^{*}\Omega^{1}_{Y}(\log D')\to \Omega^{1}_{Z}(\log E_{Z})\right). 
\end{equation} 
Without loss of generality, we may assume that 
$D$ is smooth and that $\Omega^{1}_{X/Y}(\log E)$ and 
$\Omega^{1}_{Z/Y}(\log E_{Z})$ are both locally free sheaves. 
By construction, we see that $D\leq D'$ holds. 
We consider 
the Koszul filtration 
\begin{equation}
\label{p-eq4.1}
\mathrm{Koz}^{q}\Omega^{i}_{X}(\log E):=
\image\left(f^{*}\Omega^{q}_{Y}(\log D)\otimes \Omega^{i-q}_{X}
(\log E)\to \Omega^{i}_{X}(\log E)\right). 
\end{equation} 
Then we have 
\begin{equation}
\mathrm{Koz}^{q}/\mathrm{Koz}^{q+1}\Omega^{i}_{X}(\log E)
\simeq f^{*}\Omega^{q}_{Y}(\log D)
\otimes \Omega^{i-q}_{X/Y}(\log E) 
\end{equation} 
and get the following short exact sequence: 
\begin{equation}
0\to f^{*}\Omega^{1}_{Y}(\log D)\otimes 
\Omega^{i-1}_{X/Y}(\log E)\to 
\mathrm{Koz}^{0}/\mathrm{Koz}^{2}\Omega^{i}_{X}(\log E)
\to \Omega^{i}_{X/Y}(\log E)\to 0, 
\end{equation} 
which is denoted by $\mathscr C^{i}_{X/Y}(\log E)$. 
Similarly, we can define 
$\mathrm{Koz}^{q}\Omega^{i}_{Z}(\log E_{Z})$ and 
obtain $\mathscr C^{i}_{Z/Y}(\log E_{Z})$. 
By construction, we have a natural map 
$\psi^{*}\mathscr C^{i}_{X/Y}(\log E)\to 
\mathscr C^{i}_{Z/Y}(\log E_{Z})$ for every $i$. 
By tensoring with the natural injection $\psi^*\mathscr N^{\otimes -1}
\hookrightarrow \mathscr O_Z$, we have 
\begin{equation}
\psi^*\left(\mathscr C^i_{X/Y}(\log E)\otimes 
\mathscr N^{\otimes -1}\right)\to 
\mathscr C^i_{Z/Y}(\log E_Z). 
\end{equation} 
Then, by using the edge homomorphism of the Leray spectral 
sequence, 
we obtain the natural homomorphism 
\begin{equation}
R^{d-i} f_*\left(\Omega^i_{X/Y}(\log E)\otimes \mathscr N^{\otimes -1}\right)
\to R^{d-i}h_*\left(\psi^*\left(\Omega^i_{X/Y}(\log E)
\otimes \mathscr N^{\otimes -1}\right)\right),  
\end{equation} 
where $d=\dim X-\dim Y$. Thus we get the following commutative 
diagram of the connecting homomorphisms.  
\begin{equation}
\xymatrix{
R^{d-i}f_*\left(\Omega^i_{X/Y}(\log E)
\otimes \mathscr N^{\otimes -1}\right)\ar[d]_-{\rho_{d-i}}\ar[r]& 
R^{d-i+1}f_*\left(\Omega^{i-1}_{X/Y}(\log E)
\otimes \mathscr N^{\otimes -1}\right)\otimes \Omega^1_Y(\log D)
\ar[d]^-{\rho_{d-i+1}\otimes \iota}\\ 
R^{d-i}h_*\left(\Omega^i_{Z/Y}(\log E_Z)
\right)\ar[r]_-{\phi'_{d-i}}& 
R^{d-i+1}h_*\left(\Omega^{i-1}_{Z/Y}(\log E_Z)
\right)\otimes \Omega^1_Y(\log D')
}
\end{equation}
We get a graded polarizable admissible variation of 
$\mathbb R$-mixed Hodge structure over $Y\setminus D'$ from 
$h\colon (Z, E_Z)\to (Y, D')$. 
By Theorem \ref{x-thm4.1} below, 
\begin{equation}
\bigoplus _{k=0}^d R^{k}h_*\left(\Omega^{d-k}_{Z/Y} (\log E_Z)\right)
\end{equation} 
is the lower canonical extension of the system of Hodge 
bundles associated to the above variation of $\mathbb R$-mixed 
Hodge structure. Then we put 
\begin{equation}
\mathscr F_k:=\left(\image \left(\rho_k \colon 
R^kf_*\left(\Omega^{d-k}_{X/Y}(\log E)\otimes 
\mathscr N^{\otimes -1}\right)\to 
R^kh_*\left(\Omega^{d-k}_{Z/Y}(\log E_Z)\right)\right)\right)^{\vee \vee}
\end{equation} 
for $0\leq k\leq d$. We put $\mathscr F_k=0$ if $k<0$ or $k>d=\dim 
X-\dim Y$. 
Then we see that $\mathscr F_\bullet$ is a graded logarithmic 
Higgs sheaf with poles along $D$. 
Since $\rho_0$ is the pushforward $f_*$ of the inclusion 
\begin{equation}
f^*\mathscr L=\omega_X(E)\otimes f^*\left(\omega_Y(D)\right)^{\otimes -1}
\otimes \mathscr N^{\otimes -1}\to 
\psi_* \left(\omega_Z(E_Z)
\otimes h^*\left(\omega_Y(D')\right)^{\otimes -1}\right). 
\end{equation} 
This implies that $\mathscr F_0=\left(\mathcal L\otimes f_*\mathscr 
O_X\right)^{\vee \vee}$. 
Hence $\mathscr L\subset \mathscr F_0$ holds. 
Therefore, (i), (ii), and (iii) hold. 
By taking a suitable compactification and applying 
Theorem \ref{q-thm3.1}, 
we obtain (iv). We finish the proof. 
\end{proof}

\begin{rem}[{see \cite[Remark 3.6]{park}}]\label{p-rem4.3} 
In Theorem \ref{p-thm4.2}, 
we can replace the assumption 
\begin{equation}
\kappa (X, \omega_X(E)\otimes f^{*}(\omega_{Y}(D))^{\otimes -1}
\otimes f^{*}\mathscr L^{\otimes -1})\geq 0  
\end{equation} 
with the existence of a nonzero homomorphism 
\begin{equation}\label{p-eq4.2}
\mathscr L^{\otimes N} \to 
\left( f_*\omega ^{\otimes N}_{(X, E)/(Y, D)}\right)^{\vee\vee}
\end{equation} 
for some positive integer $N$. Note that 
\eqref{p-eq4.2} implies the existence of a nonzero section 
of 
\begin{equation}
\left(
\omega_X(E)\otimes f^{*}(\omega_{Y}(D))^{\otimes -1}
\otimes f^{*}\mathscr L^{\otimes -1}\right)^{\otimes N}
\end{equation}
over the complement of some codimension two closed subset $\Sigma$ 
in $Y$. 
Hence we can construct a desired graded 
logarithmic Higgs sheaf $\mathscr F_\bullet$ on $Y\setminus \Sigma$. 
By taking the reflexive hull of $\mathscr F_\bullet$, 
we can extend $\mathscr F_\bullet$ over $Y$. 
\end{rem}

For geometric applications, the following lemma is crucial. 

\begin{lem}[{\cite[Lemma 3.7]{park}}]\label{p-lem4.4}
Let $Y$ be a smooth projective variety and let $D$ be a simple 
normal crossing divisor on $Y$. 
Let $\mathscr F_\bullet$ be a graded logarithmic 
Higgs sheaf with poles along $D$ satisfying {\em{(i)}}, {\em{(ii)}}, 
{\em{(iii)}}, and 
{\em{(iv)}} in Theorem \ref{p-thm4.2}. Then we have a pseudo-effective 
line bundle $\mathscr P$ and a nonzero homomorphism 
\begin{equation}
\mathscr L^{\otimes r}\otimes \mathscr P
\to \left(\Omega^1_Y(\log D)\right)^{\otimes kr}
\end{equation} 
for some $r>0$ and $k\geq 0$. 
\end{lem}

\begin{proof}[Proof of Lemma \ref{p-lem4.4}]
We have a sequence of homomorphisms 
\begin{equation}
\phi_k\otimes \id\colon 
\mathscr F_k\otimes \left(\Omega^1_Y(\log D)\right)^{\otimes k}
\to \mathscr F_{k+1}\otimes \left(\Omega^1_Y(\log D)\right)^{\otimes k+1}. 
\end{equation} 
Note that $\mathscr F_k$ is zero for $k\gg 0$ and 
that $\mathscr L\subset \mathscr F_0$. 
Hence, the line bundle $\mathscr L$ is contained in the kernel of 
$\phi_k\otimes \id$ for some 
$k\geq 0$, that is, 
\begin{equation}
\mathscr L\subset \mathscr K_k(\phi)
\otimes \left(\Omega^1_Y(\log D)\right)^{\otimes k}. 
\end{equation}
This implies the existence of a nonzero homomorphism 
\begin{equation}
\mathscr K_k(\phi)^{\vee} 
\to \left(\Omega^1_Y(\log D)\right)^{\otimes k} 
\otimes \mathscr L^{\otimes -1}. 
\end{equation}
Let $\mathscr Q$ be the image of the above homomorphism 
and let $r$ be the rank of $\mathscr Q$. 
By considering the split surjection 
\begin{equation}
\mathscr Q^{\otimes r}\to \det \mathscr Q
\end{equation} 
outside some suitable codimension 
two closed subset of $Y$, 
we have a nonzero homomorphism 
\begin{equation}
\widehat \det \mathscr Q\to \left(\left(\Omega^1_Y(\log D)\right)^{\otimes k} 
\otimes \mathscr L^{\otimes -1}\right)^{\otimes r}. 
\end{equation} 
Since $\mathscr P:=\widehat \det \mathscr Q$ is a pseudo-effective 
line bundle by Theorem \ref{p-thm4.2} (iv), we obtain a desired 
nonzero homomorphism. We finish the proof. 
\end{proof}

Let us prove Theorem \ref{p-thm1.4}, which is 
one of the main results of this paper. 

\begin{proof}[Proof of Theorem \ref{p-thm1.4}] 
By Remark \ref{p-rem4.3}, we can construct a graded logarithmic Higgs 
sheaf with poles along $D$ satisfying (i), (ii), (iii), and (iv) 
in Theorem \ref{p-thm4.2}. Then, by Lemma \ref{p-lem4.4}, 
we obtain a desired pseudo-effective 
line bundle and a nonzero homomorphism. 
We finish the proof. 
\end{proof}

As a direct consequence of Lemma \ref{p-lem4.4}, we have: 

\begin{thm}[{see \cite[Theorem 1.7 (2)]{park}}]\label{p-thm4.5}
Let $f\colon X\to Y$ be a surjective morphism of 
smooth projective varieties and let $E$ and $D$ be 
simple normal crossing divisors on $X$ and $Y$, respectively. 
Assume that $f$ is smooth over $Y\setminus D$, 
$E$ is relatively normal crossing over $Y\setminus D$, 
and $f^{-1}(D)\subset E$. 
In this situation, if 
$\kappa (X, K_X+E)\geq 0$ holds, then 
$K_Y+D$ is pseudo-effective.   
\end{thm}

\begin{proof}[Proof of Theorem \ref{p-thm4.5}]
The proof of \cite[Theorem 1.7 (2)]{park} works. 
We put $\mathscr L:=\omega^{\otimes -1}_{(Y, D)}$. 
By assumption, we have 
\begin{equation}
\kappa (X, \omega_{(X, E)/(Y, D)}\otimes f^*\mathscr L^{\otimes -1})
=\kappa (X, K_X+E)\geq 0. 
\end{equation} 
By Lemma \ref{p-lem4.4}, this implies that there exists a pseudo-effective 
line bundle $\mathscr P$ and a nonzero homomorphism 
\begin{equation}
\omega^{\otimes -r}_{(Y, D)}\otimes \mathscr P 
\to 
\left(\Omega^1_{Y}(\log D)\right)^{\otimes kr}
\end{equation} 
for some $r>0$ and $k\geq 0$. 
Thus $K_Y+D$ is pseudo-effective 
by \cite[Theorem 3.9]{park}, which is 
due to \cite[Theorem 7.6]{campana-paun}. 
We finish the proof. 
\end{proof}

In Section \ref{q-sec5}, we will use Theorem \ref{p-thm4.5} in the 
proofs of Corollary \ref{p-cor1.13} and 
Theorem \ref{p-thm1.14}. 

\section{Proofs}\label{q-sec5}

In this section, we will prove results in Section \ref{p-sec1}. 
Let us start with the proof of Theorem \ref{p-thm1.1}. 

\begin{proof}[Proof of Theorem \ref{p-thm1.1}]
In \cite[Section 2]{park}, this theorem 
is proved under the extra assumption that $f^{-1}(D)=E$ and $g^{-1}(D)=D'$ 
hold. However, we can easily see that 
\cite[Proposition 2.5]{park} implies 
the desired 
inclusion \eqref{p-eq1.1}. Moreover, by the proof in \cite[Section 2]{park}, 
it is easy to see that the inclusion 
\eqref{p-eq1.1} is an isomorphism over 
some nonempty Zariski open subset of $Y$. 
\end{proof}

\begin{proof}[Proof of Corollary \ref{p-cor1.2}]
This corollary is an easy consequence of Theorem \ref{p-thm1.1}. 
All we have to do is to apply Theorem \ref{p-thm1.1} 
repeatedly.  
Note that the inclusion \eqref{p-eq1.2} is an isomorphism 
over some nonempty Zariski open subset of $Y$. 
\end{proof}

We have already proved Theorem \ref{p-thm1.4} 
in Section \ref{p-sec4}. 
Thus, let us prove Theorem \ref{p-thm1.5}. 

\begin{proof}[Proof of Theorem \ref{p-thm1.5}]
If $\kappa (Y, K_Y+D)=\dim Y$, that is, $K_Y+D$ is big, 
then 
\begin{equation}
\kappa (X, K_X+E)=\kappa (F, (K_X+E)|_F)+\dim Y
\end{equation} 
holds by Maehara's 
theorem (see \cite{maehara} and \cite{fujino-notes}). 
On the other hand, if the equality 
\begin{equation}
\kappa (X, K_X+E)=\kappa (F, (K_X+E)|_F)+\dim Y
\end{equation} holds, then 
there exists a positive integer $N$ and 
an ample Cartier divisor $A$ on $Y$ such that 
\begin{equation}
f^*\mathscr O_Y(A)\subset \omega_{(X, E)}^{\otimes N}
\end{equation} 
by \cite[Proposition 1.14]{mori}. 
We put $\mathscr L:=\mathscr O_Y(A)\otimes \omega_{(Y, D)}^{\otimes -N}$. 
Hence we have 
\begin{equation}
\mathscr L^{\otimes N} 
\subset \left( \bigotimes ^N 
f_*\omega_{(X, E)/(Y, D)}^{\otimes N}\right)^{\vee\vee} 
\subset \left( f^{(N)}_*\omega_{(X^{(N)}, 
E^{(N)})/(Y, D)}^{\otimes N}\right)^{\vee\vee}.  
\end{equation} 
Here we used Corollary \ref{p-cor1.2} with $s=N$. 
By Theorem \ref{p-thm1.4}, 
there exist a pseudo-effective line bundle 
$\mathscr P$ on $Y$ and a nonzero homomorphism 
\begin{equation}
\mathscr L^{\otimes r}\otimes \mathscr P\to 
\left(\Omega^1_Y(\log D)\right)^{\otimes kr}
\end{equation} 
for some $r>0$ and $k\geq 0$. 
Hence we have a nonzero homomorphism 
\begin{equation}
\mathscr O_Y(A)^{\otimes r}\otimes \mathscr P\to 
\left(\Omega^1_Y (\log D)\right)^{\otimes kr}\otimes \omega^{\otimes Nr}_{(Y, D)}. 
\end{equation}
Then, by \cite[Theorem 3.9]{park}, which is due 
to \cite[Theorem 7.6]{campana-paun}, 
$K_Y+D$ is pseudo-effective. 
Since we have 
$\omega_{(Y, D)}\subset \left(\Omega^1_Y(\log D)\right)^{\otimes \dim Y}$ 
by definition, 
we have a nonzero homomorphism 
\begin{equation}
\mathscr O_Y(A)^{\otimes r}\otimes \mathscr P\to 
\left(\Omega^1_Y(\log D)\right)^{\otimes N'}
\end{equation} 
for some positive integer $N'$. 
Therefore, by \cite[Theorem 3.8]{park}, which is 
due to \cite[Theorem 1.3]{campana-paun}, $K_Y+D$ is the sum 
of an ample divisor and a pseudo-effective 
divisor, so it is big as desired. We finish the proof.  
\end{proof}

\begin{proof}[Proof of Corollary \ref{p-cor1.7}] 
By the easy addition formula, we have 
\begin{equation}
\dim X=\kappa (X, K_X+E)
=\kappa (F, (K_X+E)|_F)+\dim Y
\end{equation} 
and 
\begin{equation}
\kappa (F, (K_X+E)|_F)=\dim F, 
\end{equation}
where $F$ is a general fiber of $f\colon X\to Y$. 
By Theorem \ref{p-thm1.5}, we obtain 
$\kappa (Y, K_Y+D)=\dim Y$. We finish the proof. 
\end{proof}

\begin{proof}[Proof of Theorem \ref{p-thm1.8}]
We take a sufficiently large and divisible positive integer $N$ such that 
\begin{equation}
f^*\mathscr O_Y(D)\subset \omega^{\otimes N}_{(X, E)}. 
\end{equation} 
As in the proof of Theorem \ref{p-thm1.5} above, 
by Theorem \ref{p-thm1.4}, 
the proof of \cite[Theorem 1.7 (1)]{park} 
works. 
Then we obtain a positive rational number $\delta$ such that 
$K_X+(1-\delta)D$ is pseudo-effective. 
We finish the proof. 
\end{proof}

For the proof of Theorem \ref{p-thm1.12}, 
we prepare the following lemma. 
Note that we need Corollary \ref{p-cor1.2} 
in the proof of Lemma \ref{q-lem5.1}. 

\begin{lem}\label{q-lem5.1}
In Conjecture \ref{p-conj1.9}, we assume $\kappa_\sigma(Y, K_Y+D)=0$. 
Then we have 
\begin{equation}
\kappa (X, K_X+E)\leq \kappa (F, (K_X+E)|_F), 
\end{equation} 
where $F$ is a sufficiently general fiber of $f\colon X\to Y$. 
\end{lem}
\begin{proof}[Proof of Lemma \ref{q-lem5.1}] 
Note that $K_Y+D$ is pseudo-effective 
by the assumption $\kappa _\sigma(Y, K_Y+D)=0$. 
As in the proofs of Theorems \ref{p-thm1.5} and 
\ref{p-thm1.8}, 
the proof of \cite[Proposition 5.2]{park} 
works. 
We finish the proof. 
\end{proof}

Let us prove Theorem \ref{p-thm1.12}. 

\begin{proof}[Proof of Theorem \ref{p-thm1.12}]
Let $\mu\colon Y'\to Y$ be a projective birational morphism 
from a smooth 
variety $Y'$ such that $D':=\mu^{-1}(D)$ is a simple 
normal crossing divisor on $Y'$. 
We replace $(Y, D)$ with $(Y', D')$ and take suitable 
birational modifications. Then we may assume that 
the Iitaka fibration $\Phi:=\Phi_{|m(K_Y+D)|}\colon Y\to Z$ 
is a morphism onto a normal projective variety $Z$ 
with connected fibers, where 
$m$ is a sufficiently large positive integer. 
Let $(G, D|_G)$ (resp.~$(H, E|_H)$) 
be a sufficiently general fiber of $\Phi$ (resp.~$\Phi\circ f$), 
that is, $G=\Phi^{-1}(z)$ and $H=(\Phi\circ f)^{-1}(z)$, 
where $z$ is a sufficiently general point of $Z$.
Note that $G$ and $H$ are smooth projective varieties and 
$D|_G$ and $E|_H$ are simple 
normal crossing divisors. 
By construction, we see that 
$f|_H\colon H\to G$ satisfies that 
$f|_H$ is smooth over $G\setminus (D|_G)$, $E|_H$ is relatively 
normal crossing over $G\setminus (D|_G)$, and $(f|_H)^{-1}(D|_G)\subset E|_H$. 
By assumption, 
\begin{equation} 
\kappa _\sigma(G, K_G+D|_G)=
\kappa (G, K_G+D|_G)=0. 
\end{equation}  
By Lemma \ref{q-lem5.1}, 
\begin{equation}
\kappa (H, K_H+E|_H)\leq \kappa (F, (K_X+E)|_F)
\end{equation}
holds. Therefore, by applying the easy addition formula to 
$\Phi\circ f\colon X\to Z$, 
\begin{equation}
\begin{split}
\kappa (X, K_X+E)&\leq \kappa (H, K_H+E|_H)+\dim Z \\ 
&\leq \kappa (F, (K_X+E)|_F)+\kappa (Y, K_Y+D). 
\end{split} 
\end{equation} 
We finish the proof. 
\end{proof}

We need Gongyo's theorem for the proof of Corollary \ref{p-cor1.13}. 

\begin{thm}[{see 
\cite[Proposition 4.1]{fujino-subadditivity}}]\label{q-thm5.2}
Let $X$ be a smooth projective variety and let $E$ be a simple 
normal crossing divisor on $X$. 
Assume that there exists a projective birational morphism 
$\varphi\colon X\setminus E\to V$ onto an affine variety $V$. 
Then $\kappa _\sigma(X, K_X+E)=\kappa (X, K_X+E)$ 
holds, that is, the generalized abundance conjecture holds 
for $(X, E)$. In particular, if $K_X+E$ is pseudo-effective, 
then $\kappa (X, K_X+E)\geq 0$. 
\end{thm}

\begin{proof}[Proof of Theorem \ref{q-thm5.2}]
This theorem is an easy application of the minimal model program. 
For the details, see the proof of 
\cite[Proposition 4.1]{fujino-subadditivity}. 
\end{proof}

\begin{proof}[Proof of Corollary \ref{p-cor1.13}]
If $\kappa (X, K_X+E)=-\infty$, then \eqref{p-eq1.3} is 
obvious. Hence we may assume that 
$\kappa (X, K_X+E)\geq 0$. By Theorems \ref{p-thm4.5} and 
\ref{q-thm5.2}, we have $\kappa (Y, K_Y+D)\geq 0$. 
We can apply Theorem \ref{q-thm5.2} 
to a sufficiently general fiber of the Iitaka fibration 
of $Y$ with respect to $K_Y+D$. 
Therefore, by Theorem \ref{p-thm1.12}, 
we obtain the desired inequality \eqref{p-eq1.3}.  
\end{proof}

Finally, we prove Theorem \ref{p-thm1.14}. 

\begin{proof}[Proof of Theorem \ref{p-thm1.14}]
The subadditivity 
\begin{equation}
\kappa (X, K_X+E) \geq \kappa (Y, K_Y+D)+\kappa (F, (K_X+E)|_F)
\end{equation} 
follows from $\kappa _\sigma(X, K_X+E)=\kappa (X, K_X+E)$ 
(see \cite{fujino-subadditivity} and \cite{fujino-corrigendum}) 
or $\kappa _\sigma(F, (K_X+E)|_F)=\kappa (F, (K_X+E)|_F)$ 
(see \cite{hashizume}). 
Therefore, from now on, we will prove 
the superadditivity 
\begin{equation}\label{q-eq5.1}
\kappa (X, K_X+E) \leq \kappa (Y, K_Y+D)+\kappa (F, (K_X+E)|_F). 
\end{equation} 
If $\kappa (X, K_X+E)=-\infty$, then \eqref{q-eq5.1} is obviously 
true. Hence we may assume that $\kappa (X, K_X+E)\geq 0$ holds. 
By Theorem \ref{p-thm4.5}, $K_Y+D$ is pseudo-effective. 
By Conjecture \ref{p-conj1.11}, which is a special case 
of Conjecture \ref{p-conj1.10}, 
we have $\kappa (Y, K_Y+D)\geq 0$. 
Thus, by Theorem \ref{p-thm1.12} and Conjecture \ref{p-conj1.10}, 
we obtain the desired superadditivity \eqref{q-eq5.1}. 
We finish the proof. 
\end{proof}

\section{On variations of mixed Hodge structure}\label{x-sec4}

In this section, for the sake of 
completeness, we will explain the following theorem, 
which is more or less well known to the experts 
(see \cite{steenbrink-zucker}, \cite{elzein}, and so on). 
Theorem \ref{x-thm4.1} has already been used in the proof of Theorem 
\ref{p-thm4.2} and is one of the main ingredients of 
Theorem \ref{p-thm4.2}. 

\begin{thm}\label{x-thm4.1}
Let $f\colon X\to Y$ be a proper surjective morphism
from a K\"ahler manifold $X$ to a complex manifold $Y$ with
$d=\dim X-\dim Y$ and let $\Sigma_X$ and 
$\Sigma_Y$ be
reduced simple normal crossing divisors on $X$ and $Y$, respectively. 
We set $Y_0=Y \setminus \Sigma_Y, X_0=f^{-1}(Y_0), 
f_0=f|_{X_0} \colon X_0 \to Y_0$ and $U=X \setminus \Sigma_X$
and assume that $f_0$ is a smooth morphism,
$\Sigma_X \cap X_0$ is relatively normal crossing over $Y_0$,
and $\Supp f^*\Sigma_Y\subset 
\Supp \Sigma_X$ holds. 
Then the local system
$R^i(f_0|_U)_*\mathbb R_U$
underlies a graded polarizable
admissible variation of $\mathbb R$-mixed 
Hodge structure on $Y_0$ for every $i$ 
such that
the Hodge filtration $F$ on
$\mathscr V=\mathscr{O}_{Y_0} \otimes R^i(f_0|_U)_*\mathbb R_U$
extends to a filtration $F$
on the lower canonical extension ${}^\ell \mathscr V$
satisfying the following conditions:
\begin{enumerate}
\item
\label{item:7}
$\Gr_F^p\Gr_m^W({}^\ell \mathscr V)$
is locally free of finite rank for every $m,p$, and
\item
\label{item:9}
$\Gr_F^p({}^\ell \mathscr V)$
coincides with
$R^{i-p}f_*\Omega_{X/Y}^p(\log \Sigma_X)$ for every $p$,
after removing some suitable codimension two closed subset from $Y$.
\end{enumerate} 
In the above statement, if $\Sigma_X \cap X_0=0$,
then
$R^i(f_0|_U)_*\mathbb R_U=R^i(f_0)_*\mathbb R_{X_0}$
underlies a polarizable variation
of $\mathbb{R}$-Hodge structure on $Y_0$ for every $i$. 
\end{thm}

In this section, we adopt the same approach 
as in Sections 3, 4, and 7 in \cite{fujino-fujisawa3}. 
Before starting the proof of the theorem above,
we recall several facts concerning on the Koszul complex. 

\begin{say}
In the situation above,
we set $E=(f^*\Sigma_Y)_{\red}$ and $D=\Sigma_X-E$.
Then $D$ and $E$ are
reduced simple normal crossing divisors on $X$
with no common irreducible components.
The open immersions
$X \setminus D \hookrightarrow X$
and $U \hookrightarrow X_0$
are denoted by $j$ and $j_0$ respectively.
The situation is summarized in the commutative diagram
\begin{equation}
\xymatrix{
U \ar[d]_-{j_0}\ar[r]& X \setminus D \ar[d]^-j&  \\
X_0 \ar[d]_-{f_0}\ar[r]& X \ar[d]^-f& E\ar[l]\ar[d] \\
Y_0\ar[r] & Y &\Sigma_Y\ar[l]
}
\end{equation}
where the left two squares are Cartesian.

We denote by $D=\sum_{i=1}^{l}D_i$
the irreducible decomposition of $D$
and set
\begin{equation}
D^{(m)}=\coprod_{1 \le i_1 < 
\dots < i_m \le l}D_{i_1} \cap \cdots \cap D_{i_m}
\end{equation}
for $m \in \mathbb{Z}_{\ge 0}$.
(For the case of $m=0$, we set $D^{(0)}=X$ by definition.)
The natural morphism from $D^{(m)}$ to $X$ is denoted by $a_m$.

In order to define the desired weight filtration on
\begin{equation}
R^i(f_0|_U)_*\mathbb R_U
\simeq \mathbb{R} \otimes R^i(f_0|_U)_*\mathbb Q_U
\simeq \mathbb{R} \otimes R^i(f_0)_*(R(j_0)_*\mathbb Q_U)
\end{equation}
we replace $R(j_0)_*\mathbb{Q}_U$
by a Koszul complex as follows.
For the detail, see \cite[Sections 1 and 2]{Fujisawa}
(cf.~\cite{Illusie}, \cite{Steenbrink}, \cite[Section 7]{fujino-fujisawa3}).

The divisor $D$ on $X$
defines a log structure $\mathcal{M}$ by
$\mathcal{M}:=
\mathscr{O}_X
\cap
j_*\mathscr{O}^*_{X \setminus D}$
on $X$.
A morphism of abelian sheaves
$\mathscr{O}_X \to \mathcal{M}\gp$
is defined as the composite of the exponential map
$\mathscr{O}_X \ni a
\mapsto e^{2\pi\sqrt{-1}a} \in \mathscr{O}^*_X$
and the inclusion
$\mathscr{O}^*_X \hookrightarrow \mathcal{M}\gp$.
From the morphism
$\mathbf{e} \otimes \id
\colon \mathscr{O}_X \simeq \mathscr{O}_X \otimes \mathbb{Q}
\to \mathcal{M}\gp \otimes \mathbb{Q}$,
$1 \in \Gamma(X, \mathbb{Q})$
which is a global section of the kernel of $\mathbf{e} \otimes \id$,
and a subsheaf
$\mathscr{O}_X^* \otimes \mathbb{Q}
\subset \mathcal{M}\gp \otimes \mathbb{Q}$,
we obtain a complex of $\mathbb{Q}$-sheaves on $X$
\begin{equation}
\kos(\mathcal{M}):=\kos(\mathbf{e} \otimes \id; \infty ; 1)
\end{equation}
equipped with a finite increasing filtration
$W:=W(\mathscr{O}_X^* \otimes \mathbb{Q})$
as in \cite[Definition 1.8]{Fujisawa} (see also 
\cite[Section 7]{fujino-fujisawa3}).
By replacing $\mathcal{M}\gp$ by $\mathscr{O}_{X \setminus D}^*$,
we obtain a complex of $\mathbb{Q}$-sheaves on $X \setminus D$,
denoted by $\kos(\mathscr{O}_{X \setminus D}^*)$,
by the same way as above.
Moreover,
we have a morphism of complexes of $\mathbb{Q}$-sheaves
\begin{equation}
\psi \colon \kos(\mathcal{M}) \to \Omega_X(\log D)
\end{equation}
as in \cite[(2.4)]{Fujisawa},
which preserves the filtration $W$ on the both sides.
\end{say}

The following two lemmas
are more or less the same as Lemmas 7.6 and 7.7
in \cite{fujino-fujisawa3}.

\begin{lem}
\label{lem:2}
There exists a quasi-isomorphism of complexes
\begin{equation}
\label{eq:1}
(a_m)_*\mathbb{Q}_{D^{(m)}}[-m]
\to
\Gr_m^W\kos(\mathcal{M})
\end{equation}
for all $m \in \mathbb{Z}_{\ge 0}$.
\end{lem}

\begin{lem}
\label{lem:1}
The quasi-isomorphism \eqref{eq:1} makes the diagram 
\begin{equation}
\xymatrix{
(a_m)_*\mathbb{Q}_{D^{(m)}}[-m]
\ar[d]_-{(2\pi \sqrt{-1})^{-m}\iota[-m]}\ar[r]
&
\Gr_m^W\kos(\mathcal{M}) \ar[d]^-{\Gr_m^W\psi}\\
(a_m)_*\Omega_{D^{(m)}}[-m]\ar[r]
& 
\Gr_m^W\Omega_X(\log D)
}
\end{equation}
commutative,
where the bottom arrow is the inverse of the usual
residue isomorphism
and $\iota$ is the composite of the natural morphisms
$\mathbb{Q}_{D^{(m)}} \hookrightarrow \mathbb{C}_{D^{(m)}}$
and $\mathbb{C}_{D^{(m)}} \to \Omega_{D^{(m)}}$.
Consequently, the morphism $\psi$ induces a filtered quasi-isomorphism
\begin{equation}
\mathbb{C} \otimes_{\mathbb{Q}} \kos(\mathcal{M})
\to
\Omega_X(\log D)
\end{equation}
with respect to the filtration $W$ on the both sides.
\end{lem}

\begin{say}
Since $\kos(\mathcal{M})|_{X \setminus D}=\kos(\mathscr{O}^*_{X \setminus D})$,
we obtain a morphism of the complexes of $\mathbb{Q}$-sheaves 
$\kos(\mathcal{M}) \to Rj_*\kos(\mathscr{O}^*_{X \setminus D})$ 
such that the diagram in the derived category 
\begin{equation}
\label{eq:2}
\xymatrix{ 
\kos(\mathcal{M})\ar[d]_-\psi\ar[r]
&Rj_*\kos(\mathscr{O}^*_{X \setminus D}) \ar[d]\\
\Omega_X(\log D)\ar[r]
& Rj_*\Omega_{X \setminus D}
}
\end{equation}
is commutative,
where the right vertical arrow
is induced from the morphism
$\kos(\mathscr{O}^*_{X \setminus D}) \to \Omega_{X \setminus D}$
defined by the same way as $\psi$.
\end{say}

\begin{lem}
\label{lem:5}
We have the natural isomorphism
\begin{equation}
\kos(\mathcal{M}) \xrightarrow{\simeq} Rj_*\mathbb{Q}_{X \setminus D}
\end{equation}
in the derived category.
By restricting it to $X_0$, we obtain
the isomorphism
\begin{equation}
\kos(\mathcal{M})|_{X_0} \xrightarrow{\simeq} R(j_0)_*\mathbb{Q}_U
\end{equation}
in the derived category.
\end{lem}
\begin{proof}
It is sufficient to prove that the morphism
\begin{equation}
\mathbb{C} \otimes_{\mathbb{Q}} \kos(\mathcal{M})
\to
\mathbb{C} \otimes_{\mathbb{Q}}Rj_*\kos(\mathscr{O}^*_{X \setminus D})
\end{equation}
is an isomorphism.
Since we have the canonical quasi-isomorphism
$\mathbb{Q}_{X \setminus D} \to \kos(\mathscr{O}^*_{X \setminus D})$,
the right vertical arrow in \eqref{eq:2}
induces an isomorphism
$\mathbb{C} \otimes Rj_*\kos(\mathscr{O}^*_{X \setminus D})
\simeq Rj_*\Omega_{X \setminus D}$
in the derived category.
Hence Lemma \ref{lem:1} implies
the conclusion
because the bottom arrow in \eqref{eq:2}
is known to be isomorphisms
in the derived category.
\end{proof}

Now, we prove Theorem \ref{x-thm4.1}.

\begin{proof}[Proof of Theorem \ref{x-thm4.1}]
We set $E \cap D^{(m)}=(a_m)^*E$,
which is a simple normal crossing divisor on $D^{(m)}$
for every $m \in \mathbb{Z}_{\ge 0}$.

First, we assume that the following conditions are satisfied:
\begin{enumerate}
\refstepcounter{equation}
\item
\label{item:11}
$\Sigma_Y$ is a smooth hypersurface in $Y$, and
\item
\label{item:13}
$\Omega^1_{D^{(m)}/Y}(\log E \cap D^{(m)})$
is locally free of finite rank for all $m \in \mathbb{Z}_{\ge 0}$.
\end{enumerate}
On the log de Rham complex $\Omega_X(\log \Sigma_X)$,
we have the filtration $W(D)$,
which induces a finite increasing filtration $W(D)$
on the relative log de Rham complex $\Omega_{X/Y}(\log \Sigma_X)$.
A morphism of complexes
$\overline{\psi} \colon \kos(\mathcal{M}) \to \Omega_{X/Y}(\log \Sigma_X)$
is obtained by composing the three morphisms,
$\psi \colon \kos(\mathcal{M}) \to \Omega_X(\log D)$,
the inclusion $\Omega_X(\log D) \to \Omega_X(\log \Sigma_X)$
and $\Omega_X(\log \Sigma_X) \to \Omega_{X/Y}(\log \Sigma_X)$.
We set
\begin{equation}
\begin{split}
K&=((K_{\mathbb{R}}, W), (K_{\mathscr{O}}, W, F), \alpha) \\
&=(\mathbb{R} \otimes (Rf_*\kos(\mathcal{M}), W)|_{Y_0},
(Rf_*\Omega_{X/Y}(\log \Sigma_X), W(D), F),
\id \otimes (Rf_*\overline{\psi})|_{Y_0}),
\end{split}
\end{equation}
which is a triple as in \cite[3.7]{fujino-fujisawa3} on $Y$.
Since
\begin{equation}
(\Gr_m^{W(D)}\Omega_{X/Y}(\log \Sigma_X), F)
\simeq
(a_m)_*\left(\Omega_{D^{(m)}/Y}(\log E \cap D^{(m)})[-m], F[-m]\right)
\end{equation}
as filtered complexes
by Lemmas \ref{lem:2} and \ref{lem:1},
we have 
\begin{equation}
\label{eq:11}
\begin{split}
(\mathbb{R} \otimes \Gr_m^W&\kos(\mathcal{M}),
(\Gr_m^{W(D)}\Omega_{X/Y}(\log \Sigma_X), F),
\id \otimes \Gr_m^W\overline{\psi}) \\
&\simeq
(a_m)_*\left(\mathbb{R}_{D^{(m)}},
(\Omega_{D^{(m)}/Y}(\log E \cap D^{(m)}), F[-m]),
(2\pi\sqrt{-1})^{-m}\overline{\iota}\right)[-m]
\end{split}
\end{equation} 
for every $m$,
where $\overline{\iota}$
denote the composite of the canonical morphisms
\begin{equation}
\mathbb{R}_{D^{(m)}}
\hookrightarrow
\mathbb{C}_{D^{(m)}}
\to \Omega_{D^{(m)}}
\hookrightarrow
\Omega_{D^{(m)}}(\log E \cap D^{(m)})
\to
\Omega_{D^{(m)}/Y}(\log E \cap D^{(m)}).
\end{equation}
Thus we can easily check that $K|_{Y_0}$ satisfies
the conditions (3.7.1)--(3.7.3) in \cite[3.7]{fujino-fujisawa3}
because
$f a_m \colon D^{(m)} \to Y$
is smooth over $Y_0$ for every $m$.
By Lemma 3.3 together with (3.7.6) of \cite{fujino-fujisawa3},
we obtain a polarizable variation of $\mathbb{R}$-Hodge structure
\begin{equation}
\left.\left(\mathbb{R} \otimes \Gr_m^WR^if_*\kos(\mathcal{M}),
(\Gr_m^{W(D)}R^if_*\Omega_{X/Y}(\log \Sigma_X),F),
\id \otimes \Gr_m^WR^if_*\overline{\psi}\right)\right|_{Y_0}
\end{equation}
of weight $m+i$ for every $i,m$.
Using the Koszul filtration \eqref{p-eq4.1} 
in the proof of Theorem \ref{p-thm4.2},
we can check the Griffiths transversality
for $(R^if_*\Omega_{X/Y}(\log \Sigma_X),F)|_{Y_0}$
by the same way as in \cite{katzoda}
(cf.~\cite[Lemma 4.5]{fujino-fujisawa1}).
Thus the triple
\begin{equation}
\left((\mathbb{R} \otimes R^if_*\kos(\mathcal{M}), W[i]),
(R^if_*\Omega_{X/Y}(\log \Sigma_X),W(D)[i],F),
\id \otimes R^if_*\overline{\psi}\right)|_{Y_0}
\end{equation}
is a graded polarizable variation
of $\mathbb{R}$-mixed Hodge structure on $Y_0$
by \cite[(3.7.5)]{fujino-fujisawa3},
and all the local monodromies of $R^if_*\kos(\mathcal{M})|_{Y_0}$
along $\Sigma_Y$ are quasi-unipotent
by \cite[(3.7.4)]{fujino-fujisawa3}.
Since
$R^if_*\kos(\mathcal{M})|_{Y_0} \simeq R^i(f_0|_U)_*\mathbb{Q}_U$
by Lemma \ref{lem:5},
the local system
$R^i(f_0|_U)_*\mathbb{R}_U$ is of quasi-unipotent local monodromy
along $\Sigma_Y$
and underlies
a graded polarizable variation
of $\mathbb{R}$-mixed Hodge structure on $Y_0$.

Moreover, $K$ satisfies all the assumptions
in Theorem 3.9 of \cite{fujino-fujisawa3}
by \eqref{eq:11}
and by \cite[Lemma 4.3]{fujino-fujisawa3}.
Therefore, there exist isomorphisms
\begin{gather}
R^if_*\Omega_{X/Y}(\log \Sigma_X)
\simeq
{}^\ell R^if_*\Omega_{X/Y}(\log \Sigma_X)|_{Y_0} \\
W(D)_mR^if_*\Omega_{X/Y}(\log \Sigma_X)
\simeq
{}^\ell W(D)_mR^if_*\Omega_{X/Y}(\log \Sigma_X)|_{Y_0}
\end{gather}
whose restriction to $Y_0$ coincide with the identities
and the natural isomorphisms
\begin{gather}
F^pR^if_*\Omega_{X/Y}(\log \Sigma_X)
\simeq
R^if_*F^p\Omega_{X/Y}(\log \Sigma_X) \\
\Gr_F^pR^if_*\Omega_{X/Y}(\log \Sigma_X)
\simeq
R^{i-p}f_*\Omega_{X/Y}^p(\log \Sigma_X)
\end{gather}
by (3.9.1) and by (3.9.3) of \cite[Theorem 3.9]{fujino-fujisawa3},
and $\Gr_F^p\Gr_m^{W(D)}R^if_*\Omega_{X/Y}(\log \Sigma_X)$
is locally free of finite rank
for every $i,m,p \in \mathbb{Z}$
by (3.9.4) of \cite[Theorem 3.9]{fujino-fujisawa3}.
Thus the filtration $F$ on $R^if_*\Omega_{X/Y}(\log \Sigma_X)$
satisfies \ref{item:7} and \ref{item:9}.

For the general case,
by Lemma 4.6 of \cite{fujino-fujisawa3},
there exists a closed subspace $\Sigma'_Y \subset \Sigma_Y$
with $\codim_Y \Sigma'_Y \ge 2$
such that
$f \colon X \to Y$
restricted over $Y \setminus \Sigma'_Y$
satisfies the conditions
\ref{item:11} and \ref{item:13}.
Therefore the filtration $F$ on
${}^\ell \mathscr{V}|_{Y \setminus \Sigma'_Y}$
is obtained by the argument above.
Moreover, the filtration $F$ on
$\Gr_m^{W(D)} \mathscr{V}$
extends to
$\Gr_m^{W(D)}({}^\ell \mathscr{V})$
by Schmid's nilpotent orbit theorem.
Applying Lemma 1.11.2 of \cite{kashiwara},
we obtain an extension of $F$ on
${}^\ell \mathscr{V}$
satisfying \ref{item:7} and \ref{item:9}.

In order to prove the admissibility,
we may assume $(Y, \Sigma_Y)=(\Delta, \{0\})$
by the definition of admissibility (cf.~\cite[1,9]{kashiwara}).
Pulling back the variation
by the morphism
\begin{equation}
\label{eq:3}
\Delta \ni t \mapsto t^m \in \Delta
\end{equation}
changes the logarithm of the unipotent part of the monodromy automorphism
to its multiple by $m$. 
Therefore the existence of the relative monodromy weight filtration
can be checked
after the pull-back by the morphism \eqref{eq:3}.
Moreover, Lemma 1.9.1 of \cite{kashiwara}
enables us to check the extendability of the Hodge filtration
after the pull-back by the morphism \eqref{eq:3}. 
Thus we may
assume that
$f \colon X \to \Delta$
satisfies the following three conditions:
\begin{itemize}
\item 
$X$ is a K\"ahler manifold, 
\item
$f^{-1}(0)_{\red}$
is a simple normal crossing divisor on $X$, and 
\item
the local system 
$R^i(f_0|_U)_*\mathbb R_U$,
which underlies the variation of mixed Hodge structure in question,
is of unipotent monodromy automorphism around the origin.
\end{itemize}
Then we obtain the conclusion by 
\cite[Th\'eor\`eme \RomI.1.10 and Proposition \RomI.3.10]{elzein}
(cf.~\cite[\S 5]{steenbrink-zucker}, 
\cite[Theorem 14.51]{Peters-SteenbrinkMHS},
\cite[Theorem 8.2.13]{BrosnanElZeinHT}, and so on).
\end{proof}

\begin{rem}
As in \cite[Section 7]{fujino-fujisawa3}, 
we can use Koszul complexes when we construct a cohomological 
$\mathbb Q$-mixed Hodge complex for the proof of 
the above admissibility. 
\end{rem}

\section{Semipositivity theorems}\label{sec-fujisawa}

In this section, we give a variant of semipositivity theorems
of Fujita--Zucker--Kawamata
(cf.~e.g.~\cite{brunebarbe1}, \cite{brunebarbe}, \cite{fujino-fujisawa1},
\cite{fujino-fujisawa2}, \cite{fujino-fujisawa-saito},
\cite{kawamata1}, \cite{Zucker}, \cite{zuo}),
which is necessary
for the proof of Theorem \ref{thm-fujisawa}.
The arguments here are essentially the same as,
but simpler than that in \cite{fujisawa2}.

\begin{say}
Apparently, Theorem \ref{thm-fujisawa} is a direct consequence
of \cite[Theorem 4.5]{brunebarbe1}.
However, Theorem 4.5 of \cite{brunebarbe1} is not correct
because there exists a counter example
(see \cite[Example 4.6]{fujisawa2}).
In this section, we formulate another semipositivity theorem,
which is slightly different from \cite[Theorem 4.5]{brunebarbe1},
and give the proof of it
by adapting the idea in \cite{brunebarbe1}
to our formulation.
For more detail, see \cite{fujisawa2}.
\end{say}

\begin{say}
Let $X$ be a smooth complex variety.
A filtered vector bundle is a pair
$(\mathscr{V}, F)$ consisting of
a locally free $\mathscr{O}_X$-module $\mathscr{V}$ of finite rank
and a finite decreasing filtration $F$ on $\mathscr{V}$,
such that $\Gr_F^p\mathscr{V}=F^p\mathscr{V}/F^{p+1}\mathscr{V}$
is locally free
for every $p \in \mathbb{Z}$.
The notion of a morphism of filtered vector bundles
is defined in the trivial way.
For a filtered vector bundle $(\mathscr{V}, F)$,
we set $\Gr_F^{\bullet}\mathscr{V}=\bigoplus_p \Gr_F^p\mathscr{V}$.
Similarly, a bifiltered vector bundle is a triple
$(\mathscr{V}, W, F)$ consisting of
a locally free $\mathscr{O}_X$-module $\mathscr{V}$ of finite rank,
a finite increasing filtration $W$,
and a finite decreasing filtration $F$
such that
$\Gr_F^p\Gr_m^W\mathscr{V} \simeq \Gr_m^W\Gr_F^p\mathscr{V}$
is locally free
for every $m,p \in \mathbb{Z}$.

It is said that
a filtered vector bundle $(\mathscr{V}, F)$ on $X$
underlies a (polarizable) variation
of $\mathbb{R}$-Hodge structure of weight $w$,
if there exist
a polarizable variation of $\mathbb{R}$-Hodge structure
$(\mathbb{V}, F)$ of weight $w$
and an isomorphism
$(\mathscr{V}, F) \simeq (\mathscr{O}_X \otimes_{\mathbb{R}} \mathbb{V}, F)$
as filtered vector bundles.
\end{say}

\begin{say}
Let $X$ be a smooth complex variety
and $D$ a simple normal crossing divisor on $X$.
We set $X_0=X \setminus D$.
A variation of $\mathbb{R}$-Hodge structure
$(\mathbb{V}, F)$ of weight $w$ on $X_0$
is said to be unipotent,
if the monodromy automorphism
around each irreducible component of $D$
is unipotent.
Let $(\mathscr{V}, F)$ be a filtered vector bundle on $X$
such that $(\mathscr{V}, F)|_{X_0}$
underlies a unipotent variation of $\mathbb{R}$-Hodge structure
$(\mathbb{V}, F)$ of weight $w$ on $X_0$.
We call $(\mathscr{V}, F)$ the canonical extension of $(\mathbb{V}, F)$
if there exists an integrable log connection
\begin{equation}
\nabla \colon \mathscr{V}
\to \Omega_X^1(\log D) \otimes_{\mathscr{O}_X} \mathscr{V}
\end{equation}
such that the residue morphism of $\nabla$
along each irreducible component of $D$ is nilpotent
and that $\nabla|_{X_0} \simeq d \otimes \id$
under the isomorphisms
$\mathscr{V}|_{X_0}
\simeq \mathscr{O}_X \otimes_{\mathbb{R}} \mathbb{V}$
and
$(\Omega_X^1(\log D) \otimes_{\mathscr{O}_{X_0}} \mathscr{V})|_{X_0}
\simeq \Omega_{X_0}^1 \otimes_{\mathbb{R}} \mathbb{V}$.
We note that the integrable log connection $\nabla$ above
satisfies the condition
\begin{equation}
\nabla (F^p\mathscr{V}) \subset
\Omega_X^1(\log D) \otimes_{\mathscr{O}_X} F^{p-1}\mathscr{V}
\end{equation}
for every $p \in \mathbb{Z}$
because of the Griffiths transversality for $(\mathbb{V}, F)$ on $X_0$.
Then a morphism
\begin{equation}
\theta^p \colon \Gr_F^p\mathscr{V}
\to \Omega_X^1(\log D) \otimes_{\mathscr{O}_X} \Gr_F^{p-1}\mathscr{V}
\end{equation}
is induced for every $p$.
Thus the morphism defined by
\begin{equation}
\theta=\bigoplus_p \theta^p \colon
\Gr_F^{\bullet}\mathscr{V}
\to \Omega_X^1(\log D) \otimes_{\mathscr{O}_X} \Gr_F^{\bullet}\mathscr{V}
\end{equation}
is called the Higgs field of $(\mathscr{V}, F)$
as in Section \ref{q-sec2}.

In the situation above,
the filtered vector bundle $(\mathscr{V}, F)$
is called the canonical extension
of its restriction $(\mathscr{V}, F)|_{X_0}$
when the variation of $\mathbb{R}$-Hodge structure
$(\mathbb{V}, F)$ is not explicitly specified.
\end{say}

\begin{say}\label{7474}
From now, we study a bifiltered vector bundle
$(\mathscr{V},W,F)$ on $X$
satisfying the following two conditions:
\begin{enumerate}
\item
\label{item:pvhs}
$(\Gr_m^W\mathscr{V}, F)|_{X_0}$ underlies
a unipotent polarizable variation of $\mathbb{R}$-Hodge structure
of a certain weight on $X_0$
for every $m \in \mathbb{Z}$.
\item
\label{item:extension}
$(\Gr_m^W\mathscr{V}, F)$ is the canonical extension
of $(\Gr_m^W\mathscr{V}, F)|_{X_0}$ for every $m \in \mathbb{Z}$.
\end{enumerate}
For a bifiltered vector bundle $(\mathscr{V},W,F)$
satisfying \ref{item:pvhs} and \ref{item:extension},
the Higgs field of $(\Gr_m^W\mathscr{V},F)$
is denoted by $\theta_m$.
The composite
\begin{equation}
\label{eq:9}
\Gr_F^{\bullet}\Gr_m^W\mathscr{V}
\xrightarrow{\theta_m}
\Omega^1_X(\log D) \otimes_{\mathscr{O}_X}
\Gr_F^{\bullet}\Gr_m^W\mathscr{V}
\rightarrow
\Omega^1_X(\log D) \otimes_{\mathscr{O}_X}
\Gr_m^W\Gr_F^{\bullet}\mathscr{V}
\end{equation}
is denoted by $\tilde{\theta}_m$ for every $m$,
where the second arrow
is induced from the canonical isomorphism
$\Gr_F^p\Gr_m^W\mathscr{V}
\xrightarrow{\simeq} \Gr_m^W\Gr_F^p\mathscr{V}$
for all $p$.
We often omit the subscript $m$
for $\theta_m$ and for $\tilde{\theta}_m$
if there is no danger of confusion.
\end{say}

\begin{rem}\label{7575}
If a bifiltered vector bundle $(\mathscr{V},W,F)$
is the canonical extension
of a graded polarizable admissible variation of
$\mathbb{R}$-mixed Hodge structure,
then it satisfies the conditions
\ref{item:pvhs} and \ref{item:extension}.
\end{rem}

\begin{thm}
\label{thm:1}
Let $X$ be a smooth projective variety,
$D$ a simple normal crossing divisor on $X$,
and $(\mathscr{V},W,F)$ a bifiltered vector bundle on $X$
satisfying \ref{item:pvhs} and \ref{item:extension}.
Let $\mathscr{F}$ be a locally free $\mathscr{O}_X$-module
equipped with a surjection $\Gr_F^{\bullet}\mathscr{V} \to \mathscr{F}$.
The filtration $W$ on $\mathscr{V}$ induces
a filtration $W$ on $\Gr_F^{\bullet}\mathscr{V}$,
and then the filtration $W$ on $\mathscr{F}$
is induced via the surjection $\Gr_F^{\bullet}\mathscr{V} \to \mathscr{F}$.
Then $\mathscr{F}$ is semipositive
if the composite
\begin{equation}
\label{eq:6}
\Gr_F^{\bullet}\Gr_m^W\mathscr{V}
\xrightarrow{\tilde{\theta}_m}
\Omega_X^1(\log D) \otimes_{\mathscr{O}_X}
\Gr_m^W\Gr_F^{\bullet}\mathscr{V} \\
\to
\Omega_X^1(\log D) \otimes_{\mathscr{O}_X}
\Gr_m^W\mathscr{F}
\end{equation}
is the zero morphism for every $m \in \mathbb{Z}$,
where the second morphism
is induced from the surjection
$\Gr_F^{\bullet}\mathscr{V} \to \mathscr{F}$.
\end{thm}

The rest of this section is devoted to the proof of this theorem.

\begin{say}
\label{para:1}
Let us assume $\dim X \ge 2$.
We take an irreducible component of $D$ denoted by $Y$,
and set $E=D-Y$.
Then $Y$ is a smooth projective variety
and $E|_Y$ is a simple normal crossing divisor on $Y$.
We set $Y_0=Y \setminus E|_Y$.

We consider the restriction
$\mathscr{V}_Y=\mathscr{O}_Y \otimes_{\mathscr{O}_X} \mathscr{V}$
with the induced filtrations $W$ and $F$.
The residue $\res_Y(\nabla_m)$
of the associated integrable log connection $\nabla_m$
on $\Gr_m^W\mathscr{V}$
is a nilpotent endomorphism of
$\Gr_m^W\mathscr{V}_Y
\simeq \mathscr{O}_Y \otimes_{\mathscr{O}_X} \Gr_m^W\mathscr{V}$
by the assumption \ref{item:extension}.
Then the monodromy weight filtration for $\res_Y(\nabla_m)$
on $\Gr_m^W\mathscr{V}_Y$
is denoted by $W(Y)$ for every $m$.
Via the surjection 
$\pi_m \colon W_m\mathscr{V}_Y \to \Gr_m^W\mathscr{V}_Y$,
we obtain a sequence of coherent $\mathscr{O}_Y$-modules
\begin{equation}
W_{m-1}\mathscr{V}_Y \subset \cdots \subset
\pi_m^{-1}(W(Y)_{l-1}\Gr_m^W\mathscr{V}_Y)
\subset
\pi_m^{-1}(W(Y)_l\Gr_m^W\mathscr{V}_Y)
\subset \cdots \subset
W_m\mathscr{V}_Y
\end{equation}
of finite length for all $m$.
By renumbering the $\mathscr{O}_Y$-submodules
$\pi_m^{-1}(W(Y)_l\Gr_m^W\mathscr{V}_Y)$ of $\mathscr{V}_Y$,
we obtain a finite increasing filtration $M$ on $\mathscr{V}_Y$
satisfying the following two conditions:
\begin{enumerate}
\item
\label{item:1}
There exists a strictly increasing map
$\chi \colon \mathbb{Z} \to \mathbb{Z}$
such that $M_{\chi(m)}\mathscr{V}_Y=W_m\mathscr{V}_Y$
for every $m$.
\item
\label{item:2}
For every $m$,
the filtration on $\Gr_m^W\mathscr{V}_Y$
induced from $M$
coincides with $W(Y)$ up to a shift.
\end{enumerate}
Below, we will see that $(\mathscr{V}_Y,M,F)$
is a bifiltered vector bundle on $Y$
satisfying the conditions
\ref{item:pvhs} and \ref{item:extension}.
\end{say}

\begin{say}
\label{para:3}
First, we consider the pure case,
that is, the case where $W_m\mathscr{V}=V, W_{m-1}\mathscr{V}=0$
for some $m$.
Then $(\mathscr{V}, F)$ is a filtered vector bundle on $X$
such that $(\mathscr{V},F)|_{X_0}$
underlies a unipotent polarizable variation of
$\mathbb{R}$-Hodge structure $(\mathbb{V}, F)$
of weight $w$ on $X_0$,
and that $(\mathscr{V},F)$
is the canonical extension of $(\mathbb{V},F)$.
This case was already studied in
Section 5 of \cite{fujino-fujisawa1}.
Here we briefly recall some results,
which will be needed later.

The associated integrable log connection on $\mathscr{V}$
is denoted by $\nabla$.
As mentioned in \ref{para:1} above,
the nilpotent endomorphism $\res_Y(\nabla)$
gives us the monodromy weight filtration $W(Y)$ on $\mathscr{V}_Y$.
In this case, we set $M=W(Y)$ on $\mathscr{V}_Y$.

The canonical morphism
$\Omega^1_X(\log E) \to \Omega^1_X(\log D)$
induces a morphism of $\mathscr{O}_Y$-modules
$\mathscr{O}_Y \otimes_{\mathscr{O}_X} \Omega^1_X(\log E)
\to \mathscr{O}_Y \otimes_{\mathscr{O}_X}\Omega^1_X(\log D)$,
which factors through the surjection
$\mathscr{O}_Y \otimes_{\mathscr{O}_X} \Omega^1_X(\log E)
\to \Omega^1_Y(\log E|_Y)$.
Thus we obtain a morphism of $\mathscr{O}_Y$-modules
$\Omega^1_Y(\log E|_Y)
\to \mathscr{O}_Y \otimes_{\mathscr{O}_X}\Omega^1_X(\log D)$,
which fits in the short exact sequence
\begin{equation}
\label{eq:12}
\xymatrix{
0 \ar[r]
&\Omega^1_Y(\log E|_Y) \ar[r]
&\mathscr{O}_Y \otimes_{\mathscr{O}_X}\Omega^1_X(\log D) \ar[r]
&\mathscr{O}_Y \ar[r]
&0
}
\end{equation}
as in 5.14 of \cite{fujino-fujisawa1}.
By restricting $\nabla$ to $Y$,
we obtain a $\mathbb{C}$-morphism
$\mathscr{V}_Y \to \Omega^1_X(\log D) \otimes_{\mathscr{O}_X}\mathscr{V}_Y$
denote by $\nabla|_Y$.
Then we have a commutative diagram
\begin{equation}
\label{eq:10}
\vcenter{
\xymatrix{
& & \mathscr{V}_Y \ar@{=}[r] \ar[d]_-{\nabla|_Y}
& \mathscr{V}_Y \ar[d]^-{\res_Y(\nabla)} \\
0 \ar[r]
&\Omega^1_Y(\log E|_Y) \otimes_{\mathscr{O}_Y} \mathscr{V}_Y \ar[r]
&\Omega^1_X(\log D) \otimes_{\mathscr{O}_X} \mathscr{V}_Y \ar[r]
&\mathscr{V}_Y \ar[r]
&0
}}
\end{equation}
by the definition of $\res_Y(\nabla)$,
where the bottom row,
which is induced from \eqref{eq:12},
is exact.
From the local description in \ref{para:2} below,
$\nabla|_Y$ preserves the filtration $W(Y)$,
that is,
\begin{equation}
(\nabla|_Y)(W(Y)_l\mathscr{V}_Y)
\subset
\Omega^1_X(\log D) \otimes_{\mathscr{O}_X} W(Y)_l\mathscr{V}_Y
\end{equation}
for every $l$.
Therefore the morphism
\begin{equation}
\Gr_l^{W(Y)}\nabla|_Y \colon \Gr_l^{W(Y)}\mathscr{V}_Y
\to
\Omega^1_X(\log D) \otimes_{\mathscr{O}_X} \Gr_l^{W(Y)}\mathscr{V}_Y
\end{equation}
is induced.
Because the composite of
$\Gr_l^{W(Y)}\nabla|_Y$
and the morphism
\begin{equation}
\Omega^1_X(\log D) \otimes_{\mathscr{O}_X} \Gr_l^{W(Y)}\mathscr{V}_Y
\to \Gr_l^{W(Y)}\mathscr{V}_Y
\end{equation}
induced by the morphism in the bottom row of \eqref{eq:10}
is the zero morphism
by the commutativity of \eqref{eq:10}
and by the inclusion
$\res_Y(\nabla)(W(Y)_l\mathscr{V}_Y) \subset W(Y)_{l-2}\mathscr{V}$,
a morphism
\begin{equation}
\nabla^{(l)}_Y
\colon
\Gr_l^{W(Y)}\mathscr{V}_Y
\to \Omega^1_Y(\log E|_Y) \otimes_{\mathscr{O}_Y}
\Gr_l^{W(Y)}\mathscr{V}_Y
\end{equation}
is induced from $\Gr_l^{W(Y)}\nabla|_Y$.
In other words, we have the commutative diagram
\begin{equation}
\xymatrix{
\Gr_l^{W(Y)}\mathscr{V}_Y \ar@{=}[r] \ar[d]_-{\nabla^{(l)}_Y}
& \Gr_l^{W(Y)}\mathscr{V}_Y \ar[d]^-{\Gr_l^{W(Y)}\nabla|_Y} \\
\Omega^1_Y(\log E|_Y) \otimes_{\mathscr{O}_Y} \Gr_l^{W(Y)}\mathscr{V}_Y \ar[r]
&\Omega^1_X(\log D) \otimes_{\mathscr{O}_X} \Gr_l^{W(Y)}\mathscr{V}_Y
}
\end{equation}
for every $l$.
Since $\nabla^{(l)}_Y$ is an integrable log connection
on $\Gr_l^{W(Y)}\mathscr{V}_Y$
by the local description \ref{para:2} below,
we obtain a $\mathbb{C}$-local system
$\ker(\nabla^{(l)}_Y)|_{Y_0}$
such that
\begin{equation}
\mathscr{O}_{Y_0} \otimes \ker(\nabla^{(l)}_Y)|_{Y_0}
\simeq \Gr_l^{W(Y)}\mathscr{V}_Y|_{Y_0}
\end{equation}
for every $l$.
\end{say}

\begin{say}
\label{para:2}
In this paragraph,
we give a local description of the objects appeared in \ref{para:3}.
Let us assume that $X=\Delta^{n+1}$ with the coordinate
$(y,z_1, \dots, z_n)$,
on which the divisors $Y$ and $E$
are defined by $y$ and $z_1 \cdots z_k$ respectively
for some $k$ with $1 \le k \le n$.
By the assumption that $(\mathscr{V},F)|_{X_0}$
underlies a unipotent polarizable variation of $\mathbb{R}$-Hodge structure
$(\mathbb{V},F)$ of weight $w$ on $X_0$,
there exist a finite dimensional $\mathbb{R}$-vector space $V$,
the mutually commuting nilpotent endomorphisms $N, N_1, \dots, N_k$ of $V$,
and an isomorphism
\begin{equation}
\varpi
\colon
\mathscr{O}_X \otimes_{\mathbb{R}} V
\xrightarrow{\simeq}
\mathscr{V}
\end{equation}
such that the pull-back $\varpi^*\nabla$
on $\mathscr{O}_X \otimes_{\mathbb{R}} V$
is given by
\begin{equation}
\label{eq:13}
(\varpi^*\nabla)(f \otimes v)
=
df \otimes v+(2\pi\sqrt{-1})^{-1}f
\Bigl(
\frac{dy}{y} \otimes N(v)+\sum_{i=1}^{k}\frac{dz_i}{z_i} \otimes N_i(v)
\Bigr)
\end{equation}
for $f \in \mathscr{O}_X$ and $v \in V$.
Moreover, once we define a (multi-valued) morphism
of $\mathscr{O}_{X_0}$-modules
$\rho \colon \mathscr{O}_{X_0} \otimes_{\mathbb{R}} V
\to \mathscr{O}_{X_0} \otimes_{\mathbb{R}} V$
by
\begin{equation}
\label{eq:14}
\rho=\exp
\Bigl(
-(2\pi\sqrt{-1})^{-1}(
\log y(\id \otimes N)+\sum_{i=1}^{k}\log z_i(\id \otimes N_i))
\Bigr),
\end{equation}
then we have
$(\varpi|_{X_0})^{-1}\mathbb{V}=\rho(V)$.
We denote the isomorphism
$\mathscr{O}_Y \otimes_{\mathbb{R}} V \xrightarrow{\simeq} \mathscr{V}_Y$
induced from $\varpi$
by $\varpi_Y$.
Since $\res_Y(\nabla)$ is identified with
$(2\pi\sqrt{-1})^{-1}\id \otimes N$
under the isomorphism $\varpi_Y$,
we have
\begin{equation}
\varpi_Y^{-1}(W(Y)_l\mathscr{V}_Y)
=\mathscr{O}_Y \otimes_{\mathbb{R}} W(N)_lV,
\end{equation}
for every $l$,
where $W(N)$ denotes the monodromy weight filtration for $N$ on $V$.
Therefore $\varpi_Y$ induces an isomorphism
\begin{equation}
\varpi^{(l)}_Y
\colon
\mathscr{O}_Y \otimes_{\mathbb{R}} \Gr_l^{W(N)}V
\xrightarrow{\simeq}
\Gr_l^{W(Y)}\mathscr{V}_Y
\end{equation}
for every $l$.
Because $N_i$ commutes with $N$ for all $i=1, \dots, k$,
the filtration $W(N)$ is preserved by $N_i$ for all $i=1, \dots, k$.
Therefore $\varpi^*\nabla$
preserves the filtration $W(N)$ by \eqref{eq:13}.
Thus $\nabla|_Y$ preserves the filtration $W(Y)$ on $\mathscr{V}_Y$.
For the morphism $\nabla^{(l)}_Y$ defined in \ref{para:3},
we have
\begin{equation}
\label{eq:16}
((\varpi^{(l)}_Y)^*\nabla^{(l)}_Y)(f \otimes v)
=
df \otimes v+(2\pi\sqrt{-1})^{-1}f
\sum_{i=1}^{k}\frac{dz_i}{z_i} \otimes N_i(v)
\end{equation}
for $f \in \mathscr{O}_Y$ and $v \in \Gr_l^{W(N)}V$.
On the other hand,
the morphism $\rho$ in \eqref{eq:14}
induces a morphism
$\rho_Y \colon
\mathscr{O}_{Y_0} \otimes_{\mathbb{C}} V
\to \mathscr{O}_{Y_0} \otimes_{\mathbb{C}} V$,
which preserves the filtration $W(N)$ on $V$.
Because of the inclusion $N(W(N)_lV) \subset W(N)_{l-2}V$,
we have
\begin{equation}
\Gr_l^{W(N)}\rho_Y
=
\exp
\Bigl(-(
2\pi\sqrt{-1})^{-1}
\sum_{i=1}^{k}\log z_i(\id \otimes N_i)
\Bigr)
\end{equation}
for every $l$.
Therefore the equality 
\begin{equation}
\ker(\nabla^{(l)}_Y)|_{Y_0}
=
(\varpi^{(l)}_Y|_{Y_0})((\Gr_l^{W(N)}\rho_Y)(\mathbb{C} \otimes_{\mathbb{R}}V))
\end{equation}
holds for every $l$ by \eqref{eq:16}.
Then 
\begin{equation}
\label{eq:17}
(\varpi^{(l)}_Y|_{Y_0})((\Gr_l^{W(N)}\rho_Y)(V))
\subset (\varpi^{(l)}_Y|_{Y_0})((\Gr_l^{W(N)}\rho_Y)
(\mathbb{C} \otimes_{\mathbb{R}}V))
=\ker(\nabla^{(l)}_Y)|_{Y_0}
\end{equation}
gives us an $\mathbb{R}$-local subsystem
of $\ker(\nabla^{(l)}_Y)|_{Y_0}$.

Changing the coordinate function $z_i$ to $az_i$
for some $i\in \{1, \dots, k\}$
with $a \in \Gamma(X, \mathscr{O}^*_{X})$
yields the commutative diagram
\begin{equation}
\label{eq:15}
\vcenter{
\xymatrix{
\mathscr{O}_Y \otimes_{\mathbb{R}} \Gr_l^{W(N)}V
\ar[d]_-{\varpi^{(l)}_Y} \ar[r]^-{\simeq}
&\mathscr{O}_Y \otimes_{\mathbb{R}} \Gr_l^{W(N)}V
\ar[d]^-{\varpi'^{(l)}_Y} \\
\Gr_l^{W(Y)}\mathscr{V}_Y \ar@{=}[r]
&\Gr_l^{W(Y)}\mathscr{V}_Y
}
}
\end{equation}
where the vertical arrows
$\varpi^{(l)}_Y$ and $\varpi'^{(l)}_Y$
are the isomorphisms
given by the coordinates $(y,z_1, \dots, z_i, \dots, z_n)$
and $(y, z_1, \dots, az_i, \dots, z_n)$ respectively,
and the top horizontal arrow
is the isomorphism
\begin{equation}
\exp(-(2\pi\sqrt{-1})^{-1}(\log a) \id \otimes N_i)
\colon
\mathscr{O}_Y \otimes_{\mathbb{R}} \Gr_l^{W(N)}V
\to
\mathscr{O}_Y \otimes_{\mathbb{R}} \Gr_l^{W(N)}V
\end{equation}
by choosing a brunch of $\log a$.
\end{say}

\begin{say}
We return to the situation in \ref{para:3}.
By Lemma 5.10 and Corollary 5.13 of \cite{fujino-fujisawa1},
the triple $(\mathscr{V}_Y, W(Y), F)$
is a bifiltered vector bundle on $Y$.
Next task is to define an $\mathbb{R}$-local system
$\mathbb{V}^{(l)}$ on $Y_0$
with
$\mathbb{C} \otimes_{\mathbb{R}} \mathbb{V}^{(l)}
\simeq
\ker(\nabla^{(l)}_Y)|_{Y_0}$ for every $l$.
Such $\mathbb{V}^{(l)}$ can be constructed by gluing as follows:
In the local situation in \ref{para:2},
the $\mathbb{R}$-local subsystem \eqref{eq:17}
of $\ker(\nabla^{(l)}_Y)|_{Y_0}$
is invariant under the change of the coordinate function
$z_i$ to $az_i$ with $a \in \mathscr{O}^*_X$
by the commutativity of the diagram \eqref{eq:15}
and the equality
\begin{equation}
\Gr_l^{W(N)}\rho'_Y
=
\exp(-(2\pi\sqrt{-1})^{-1}(\log a) \id \otimes N_i) \cdot
\Gr_l^{W(N)}\rho_Y
\end{equation}
by the suitable choice of a brunch of $\log a$,
where $\rho'_Y$ denotes the morphism
defined by the same way as $\rho_Y$
from the coordinate $(y,z_1, \dots, az_i, \dots, z_n)$.
Thus we obtain an $\mathbb{R}$-local system $\mathbb{V}^{(l)}$ on $Y_0$
with the desired property
$\mathbb{C} \otimes_{\mathbb{R}} \mathbb{V}^{(l)}
\simeq \ker(\nabla^{(l)}_Y)|_{Y_0}$ for every $l$.
Then Corollary 5.13 and Proposition 5.19 of \cite{fujino-fujisawa1}
imply that
$(\mathbb{V}^{(l)}, (\Gr_l^{W(Y)}\mathscr{V}_Y, F)|_{Y_0})$
is a polarizable variation of $\mathbb{R}$-Hodge structure
of weight $w+l$ on $Y_0$.
Combined with the fact that
the residue morphism of $\nabla^{(l)}_Y$ along
every irreducible component of $E|_Y$
is nilpotent by the local description \ref{para:2},
the bifiltered vector bundle $(\mathscr{V}_Y, W(Y), F)$
satisfies the conditions \ref{item:pvhs} and \ref{item:extension}.
\end{say}

\begin{say}
Next, we discuss the general case in \ref{para:1}.
For every $k \in \mathbb{Z}$,
there exists the unique $m(k) \in \mathbb{Z}$
such that $\chi(m(k)-1) < k \le \chi(m(k))$
by \ref{item:1}.
Then we have
\begin{equation}
\label{eq:4}
W_{m(k)-1}\mathscr{V}_Y=M_{\chi(m(k)-1)}\mathscr{V}_Y
\subset
M_{k-1}\mathscr{V}_Y
\subset
M_k\mathscr{V}_Y
\subset
M_{\chi(m(k))}\mathscr{V}_Y=W_{m(k)}\mathscr{V}_Y
\end{equation}
for every $k$.
By \eqref{eq:4},
the inclusion $M_k\mathscr{V} \hookrightarrow W_{m(k)}\mathscr{V}$
induces an isomorphism
\begin{equation}
\label{eq:5}
\Gr_k^M\mathscr{V}_Y
\xrightarrow{\simeq}
\Gr_k^M\Gr_{m(k)}^W\mathscr{V}_Y,
\end{equation}
under which
the filtration $F$ on the both sides are identified.
Because there exists an integer $l$
such that
$(\Gr_k^M\Gr_{m(k)}^W\mathscr{V}_Y, F)
\simeq (\Gr_l^{W(Y)}\Gr_{m(k)}^W\mathscr{V}_Y,F)$
as filtered vector bundles by \ref{item:2},
the triple $(\mathscr{V}_Y,M,F)$
is a bifiltered vector bundle
satisfying \ref{item:pvhs} and \ref{item:extension}
for $(Y, E|_Y)$.

Next, we relate the Higgs field of $(\Gr_k^M\mathscr{V}_Y,F)$
to that of $(\Gr_{m(k)}^W\mathscr{V}, F)$.
We fix $k \in \mathbb{Z}$
and denote
the integrable log connections
associated to
$(\Gr_{m(k)}^W\mathscr{V}, F),
(\Gr_k^M\Gr_{m(k)}^W\mathscr{V}_Y,F)$,
and $(\Gr_k^M\mathscr{V}_Y, F)$ by
$\nabla, \overline{\nabla}$ and $\nabla_Y$
respectively.
Similarly, the Higgs fields of
$(\Gr_{m(k)}^W\mathscr{V}, F)$
and $(\Gr_k^M\mathscr{V}_Y, F)$
are simply denoted by $\theta$ and $\theta_Y$ respectively.
By definition, the diagram
\begin{equation}
\xymatrix{
\Gr_k^M\mathscr{V}_Y \ar[d]_-{\simeq} \ar[r]^-{\nabla_Y}
& \Omega^1_Y(\log E|_Y) \otimes_{\mathscr{O}_Y}
\Gr_k^M\mathscr{V}_Y \ar[d]^-{\simeq} \\
\Gr_k^M\Gr_{m(k)}^W\mathscr{V}_Y \ar[r]_-{\overline{\nabla}}
& \Omega^1_Y(\log E|_Y) \otimes_{\mathscr{O}_Y}
\Gr_k^M\Gr_{m(k)}^W\mathscr{V}_Y
}
\end{equation}
is commutative,
where the left vertical arrow is the isomorphism \eqref{eq:5}
and the right is the isomorphism induced by \eqref{eq:5}.

On the other hand,
the filtration $W(Y)$ on
$\Gr_{m(k)}^W\mathscr{V}_Y$
is preserved by $\nabla|_Y$,
and hence so is $M$ on $\Gr_{m(k)}^W\mathscr{V}_Y$
by \ref{item:2}.
Moreover, the diagram
\begin{equation}
\xymatrix@C=48pt{
\Gr_k^M\Gr_{m(k)}^W\mathscr{V}_Y
\ar[r]^-{\overline{\nabla}}
\ar@{=}[d]
& \Omega^1_Y(\log E|_Y) \otimes_{\mathscr{O}_Y}
\Gr_k^M\Gr_{m(k)}^W\mathscr{V}_Y \ar[d] \\
\Gr_k^M\Gr_{m(k)}^W\mathscr{V}_Y \ar[r]_-{\Gr_k^M(\nabla|_Y)}
& (\mathscr{O}_Y \otimes_{\mathscr{O}_X} \Omega^1_X(\log D))
\otimes_{\mathscr{O}_Y}
\Gr_k^M\Gr_{m(k)}^W\mathscr{V}_Y
}
\end{equation}
is commutative,
where the right vertical arrow
is induced from the injection
$\Omega^1_Y(\log E|_Y) \hookrightarrow
\mathscr{O}_Y \otimes_{\mathscr{O}_X} \Omega^1_X(\log D)$ above.
Combining these two commutative diagram,
we obtain the commutative diagram
\begin{equation}
\xymatrix@C=64pt{
\Gr_F^p\Gr_k^M\mathscr{V}_Y \ar[d]_-{\simeq} \ar[r]^-{\theta_Y}
& \Omega^1_Y(\log E|_Y) \otimes_{\mathscr{O}_Y}
\Gr_F^{p-1}\Gr_k^M\mathscr{V}_Y \ar[d] \\
\Gr_F^p\Gr_k^M\Gr_{m(k)}^W\mathscr{V}_Y
\ar[r]^-{\Gr_F^p\Gr_k^M(\nabla|_Y)} \ar[d]_-{\simeq}
& (\mathscr{O}_Y \otimes_{\mathscr{O}_X} \Omega^1_X(\log D))
\otimes_{\mathscr{O}_Y}
\Gr_F^{p-1}\Gr_k^M\Gr_{m(k)}^W\mathscr{V}_Y \ar[d] \\
\Gr_k^M\Gr_F^p\Gr_{m(k)}^W\mathscr{V}_Y \ar[r]_-{\Gr_k^M(\theta|_Y)}
& (\mathscr{O}_Y \otimes_{\mathscr{O}_X} \Omega^1_X(\log D))
\otimes_{\mathscr{O}_Y}
\Gr_k^M\Gr_F^{p-1}\Gr_{m(k)}^W\mathscr{V}_Y
}
\end{equation}
for every $p$,
where the lower vertical arrows are induced
from the exchanging isomorphism between
$\Gr_k^M\Gr_F^p$
and $\Gr_F^p\Gr_k^M$
and between $\Gr_k^M\Gr_F^{p-1}$
and $\Gr_F^{p-1}\Gr_k^M$
respectively.
Then, by Lemma \ref{lem:3} and Corollary \ref{cor:1} below,
we have the commutative diagram
\begin{equation}
\label{eq:7}
\vcenter{
\xymatrix@C=48pt{
\Gr_F^{\bullet}\Gr_k^M\mathscr{V}_Y \ar[r]^-{\tilde{\theta}_Y} \ar[d]
& \Omega^1_Y(\log E|_Y) \otimes_{\mathscr{O}_Y}
\Gr_k^M\Gr_F^{\bullet}\mathscr{V}_Y \ar[d] \\
\Gr_k^M\Gr_F^{\bullet}\Gr_{m(k)}^W\mathscr{V}_Y
\ar[r]_-{\Gr_k^M\tilde{\theta}|_Y}
& (\mathscr{O}_Y \otimes_{\mathscr{O}_X} \Omega^1_X(\log D))
\otimes_{\mathscr{O}_Y}
\Gr_k^M\Gr_{m(k)}^W\Gr_F^{\bullet}\mathscr{V}_Y,
}}
\end{equation}
where the top horizontal arrow
is the composite \eqref{eq:9} for $\theta_Y$
and the bottom horizontal arrow
is obtained by taking $\Gr_k^M$
on the restriction to $Y$ of the morphism $\tilde{\theta}$.
\end{say}

\begin{lem}
\label{lem:3}
Let $V$ be an object of an abelian category
equipped with a finite decreasing filtrations $F$
and two increasing filtrations $W,M$.
Under the canonical isomorphism
\begin{equation}
\Gr_F^p\Gr_m^WV \simeq \Gr_m^W\Gr_F^pV
\end{equation}
for every $m,p$,
the subobject $M_k\Gr_F^p\Gr_m^WV$
on the left hand side
is isomorphic to
the subobject $M_k\Gr_m^W\Gr_F^pV$
if $W_{m-1}V \subset M_kV \subset W_mV$.
\end{lem}
\begin{proof}
Under the canonical isomorphisms
\begin{equation}
\label{eq:8}
\Gr_F^p\Gr_m^WV
\simeq
\frac{F^pV \cap W_mV}{F^{p+1}V \cap W_mV+F^pV \cap W_{m-1}V}
\simeq
\Gr_m^W\Gr_F^pV
\end{equation}
we compare
$M_k\Gr_F^p\Gr_m^WV$ and $M_k\Gr_m^W\Gr_F^pV$.
We have
\begin{equation}
\begin{split}
M_k\Gr_m^WV &\cap F^p\Gr_m^WV \\
&=\frac{(M_kV \cap W_mV+W_{m-1}V) \cap (F^pV \cap W_mV+W_{m-1}V)+W_{m-1}V}
{W_{m-1}V} \\
&=\frac{(M_kV \cap W_mV+W_{m-1}V) \cap F^pV+W_{m-1}V}{W_{m-1}V}
\end{split}
\end{equation}
on $\Gr_m^WV=W_mV/W_{m-1}V$ by definition.
Therefore, under the first isomorphism in \eqref{eq:8},
\begin{equation}
\begin{split}
M_k\Gr_F^p&\Gr_m^WV \\
&\simeq
\frac{(M_kV \cap W_mV+W_{m-1}V) \cap F^pV
+F^{p+1}V \cap W_mV+F^pV \cap W_{m-1}V}
{F^{p+1}V \cap W_mV+F^pV \cap W_{m-1}V} \\
&=
\frac{M_kV \cap F^pV
+F^{p+1}V \cap W_mV+F^pV \cap W_{m-1}V}
{F^{p+1}V \cap W_mV+F^pV \cap W_{m-1}V}
\end{split}
\end{equation}
by the assumption
$W_{m-1}V \subset M_kV \subset W_mV$.
On the other hand,
we have
\begin{equation}
\begin{split}
M_k\Gr_F^pV &\cap W_m\Gr_F^pV \\
&=\frac{(M_kV \cap F^pV+F^{p+1}V) \cap (W_mV \cap F^pV+F^{p+1}V)+F^{p+1}V}
{F^{p+1}V} \\
&=\frac{(M_kV \cap F^pV+F^{p+1}V) \cap W_mV+F^{p+1}V}{F^{p+1}V}
\end{split}
\end{equation}
and then
\begin{equation}
M_k\Gr_m^W\Gr_F^pV \simeq
\frac{(M_kV \cap F^pV+F^{p+1}V) \cap W_mV+F^{p+1}V \cap W_mV+F^pV \cap W_{m-1}V}
{F^{p+1}V \cap W_mV+F^pV \cap W_{m-1}V}
\end{equation}
under the second isomorphism in \eqref{eq:8}.
By the assumption $M_kV \subset W_mV$,
the equality
\begin{equation}
(M_kV \cap F^pV+F^{p+1}V) \cap W_mV
=M_kV \cap F^pV+F^{p+1}V \cap W_mV
\end{equation}
can be easily checked.
Thus we obtain
\begin{equation}
M_k\Gr_m^W\Gr_F^pV \simeq
\frac{M_kV \cap F^pV+F^{p+1}V \cap W_mV+F^pV \cap W_{m-1}V}
{F^{p+1}V \cap W_mV+F^pV \cap W_{m-1}V}
\end{equation}
which implies the conclusion.
\end{proof}

\begin{cor}
\label{cor:1}
In the same situation as in Lemma \ref{lem:3},
we have the commutative diagram
\begin{equation}
\xymatrix{
\Gr_F^p\Gr_k^MV \ar[r]^-{\simeq} \ar[d]_-{\simeq}
& \Gr_k^M\Gr_F^pV \ar[dd]^-{\simeq} \\
\Gr_F^p\Gr_k^M\Gr_m^WV \ar[d]_-{\simeq} & \\
\Gr_k^M\Gr_F^p\Gr_m^WV \ar[r]_-{\simeq} & \Gr_k^M\Gr_m^W\Gr_F^pV
}
\end{equation}
for the canonical isomorphisms
if $W_{m-1}V \subset M_{k-1}V \subset M_kV \subset W_mV$.
\end{cor}
\begin{proof}
As in the proof of Lemma \ref{lem:3},
all objects in the diagram above
are the quotient of $M_kV \cap F^pV$.
Thus the conclusion is easily obtained.
\end{proof}

Now we are ready to prove Theorem \ref{thm:1}.

\begin{proof}[Proof of Theorem \ref{thm:1}]
What we have to show is the inequality $\deg \mathscr{L} \ge 0$
for every quotient line bundle $\mathscr{L}$ of $f^*\mathscr{F}$
and for a morphism $f \colon C \to X$ from a smooth projective curve $C$.
We prove it by induction on $\dim X$.

First we consider the case of $\dim X=1$.
Then we may assume that $f$ is surjective.
Therefore $f^*(\mathscr{V},W,F)$ satisfies
the conditions \ref{item:pvhs} and \ref{item:extension}
for $f^{-1}(X_0) \subset C$.
Thus we may assume that $X=C$ and $f=\id$.
Take the smallest integer $m$
such that the induced morphism
$W_m\mathscr{F} \to \mathscr{L}$
is not the zero morphism.
Then its image $\mathscr{M}$
admits a surjective morphism $\Gr_m^W\mathscr{F} \to \mathscr{M}$
by definition.
On the other hand, $\mathscr{M}$
is an invertible subsheaf of $\mathscr{L}$
because $\dim X=1$.
Then it is sufficient to prove $\deg \mathscr{M} \ge 0$.
By replacing $(\mathscr{V}, F)$ by $(\Gr_m^W\mathscr{V},F)$,
we may assume that $(\mathscr{V},W,F)$ is pure.
Then we can apply the argument
in the proof of Lemma 4.7 of \cite{brunebarbe1}
to the dual $\mathscr{F}^{\vee}$,
which is a subbundle of
$(\Gr_F^{\bullet}\mathscr{V})^{\vee}
\simeq \Gr_F^{\bullet}\mathscr{V}^{\vee}$
contained in the kernel of the Higgs field $\theta$,
and obtain the desired inequality $\deg \mathscr{M} \ge 0$ 
(see also Remark \ref{rem-new}).

Next, we assume $\dim X \ge 2$.
If $f(C) \cap X_0 \not= \emptyset$,
then $f^*(\mathscr{V},W,F)$ satisfies
the conditions \ref{item:pvhs} and \ref{item:extension}
for $f^{-1}(X_0) \subset C$.
In such a case, the same argument as above
implies the desired conclusion.
Therefore we may assume $f(C) \subset D$.
Then there exists an irreducible component $Y$ of $D$
with $f(C) \subset Y$.
As constructed in the paragraph \ref{para:1},
we have a bifiltered vector bundle
$(\mathscr{V}_Y,M,F)$ on $Y$
satisfying the conditions \ref{item:pvhs} and \ref{item:extension}
for $(Y,E|_Y)$.
The locally free $\mathscr{O}_Y$-module
$\mathscr{F}_Y=\mathscr{O}_Y \otimes_{\mathscr{O}_X}\mathscr{F}$
is equipped with the induced surjection
$\Gr_F^{\bullet}\mathscr{V}_Y \to \mathscr{F}_Y$.
Then we have the commutative diagram
\begin{equation}
\xymatrix{
\Gr_F^{\bullet}\Gr_k^M\mathscr{V}_Y \ar[r] \ar[d]
& \Omega^1_Y(\log E|_Y) \otimes_{\mathscr{O}_Y}
\Gr_k^M\mathscr{F}_Y \ar[d] \\
\Gr_k^M\Gr_F^{\bullet}\Gr_{m(k)}^W\mathscr{V}_Y \ar[r]
& (\mathscr{O}_Y \otimes_{\mathscr{O}_X} \Omega^1_X(\log D))
\otimes_{\mathscr{O}_Y}
\Gr_k^M\Gr_{m(k)}^W\mathscr{F}_Y
}
\end{equation}
by the diagram \eqref{eq:7}.
Therefore the composite of $\tilde{\theta}_Y$
and
\begin{equation}
\Omega_Y^1(\log E|_Y) \otimes_{\mathscr{O}_Y}
\Gr_k^M\Gr_F^{\bullet}\mathscr{V}_Y \\
\to
\Omega_Y^1(\log E|_Y) \otimes_{\mathscr{O}_Y}
\Gr_k^M\mathscr{F}
\end{equation}
as in \eqref{eq:6},
which is the top horizontal arrow in the diagram above,
is the zero morphism
because
the morphism
$\Omega^1_Y(\log E|_Y)
\to \mathscr{O}_Y \otimes_{\mathscr{O}_X}\Omega^1_X(\log D)$
is injective
and because the bottom horizontal arrow
is the zero morphism by the assumption for $\mathscr{F}$.
Therefore $\mathscr{F}_Y$ is semipositive by the induction hypothesis.
Thus $f^*\mathscr{F}=f^*\mathscr{F}_Y$
satisfies the positivity property as desired.
\end{proof}



\begin{thebibliography}{FnFsS} 

\bibitem[BE]{BrosnanElZeinHT} 
P.~Brosnan, F.~El Zein,
Variations of mixed hodge structure,
in {\em{Hodge Theory}},
Mathematical Notes \textbf{49} (2014), 333--409.
Princeton University Press. 

\bibitem[B1]{brunebarbe1} 
Y.~Brunebarbe, 
Symmetric differentials 
and variations of Hodge structures, 
J. Reine Angew. Math. \textbf{743} (2018), 133--161.

\bibitem[B2]{brunebarbe} 
Y.~Brunebarbe, 
Semi-positivity from Higgs bundles, 
preprint (2017). arXiv:1707.08495 [math.AG]

\bibitem[C]{campana} 
F.~Campana, 
Kodaira additivity, birational isotriviality, 
and specialness, Mosc. Math. J. \textbf{23} (2023), no. 3, 319--330.
 
\bibitem[CP]{campana-paun} 
F.~Campana, M.~P\u aun, 
Foliations with positive slopes and birational stability of 
orbifold cotangent bundles, 
Publ. Math. Inst. Hautes \'Etudes Sci. \textbf{129} (2019), 1--49.

\bibitem[EjFI]{ejiri-fujino-iwai} 
S.~Ejiri, O.~Fujino, M.~Iwai, 
Positivity of extensions of vector bundles, 
Math. Z. \textbf{306} (2024), no. 3, Paper no. 47. 

\bibitem[El]{elzein} 
F.~El Zein, Th\'eorie de Hodge des cycles \'evanescents, 
Ann. Sci. \'Ecole Norm. Sup. (4) \textbf{19} (1986), no. 1, 107--184.

\bibitem[Fn1]{fujino-notes} 
O.~Fujino, Notes on the weak positivity theorems, 
{\em{Algebraic varieties and automorphism groups}}, 73--118, 
Adv. Stud. Pure Math., \textbf{75}, Math. Soc. Japan, Tokyo, 2017.

\bibitem[Fn2]{fujino-foundations} 
O.~Fujino, {\em{Foundations of the minimal model program}}, 
MSJ Memoirs, \textbf{35}. Mathematical Society of Japan, Tokyo, 2017.

\bibitem[Fn3]{fujino-subadditivity} 
O.~Fujino, 
On subadditivity of the logarithmic Kodaira dimension, 
J. Math. Soc. Japan \textbf{69} (2017), no. 4, 1565--1581. 

\bibitem[Fn4]{fujino-corrigendum} 
O.~Fujino, 
Corrigendum to \lq\lq On subadditivity of the logarithmic Kodaira 
dimension\rq\rq, J. Math. Soc. Japan \textbf{72} (2020), no. 4, 1181--1187. 

\bibitem[Fn5]{fujino-iitaka} 
O.~Fujino, {\em{Iitaka conjecture---an introduction}}, 
SpringerBriefs Math.
Springer, Singapore, 2020. 

\bibitem[FnFs1]{fujino-fujisawa1}
O.~Fujino, T.~Fujisawa, 
Variations of mixed Hodge structure and semipositivity 
theorems, 
Publ. Res. Inst. Math. Sci. \textbf{50} (2014), no. 4, 589--661. 

\bibitem[FnFs2]{fujino-fujisawa2}
O.~Fujino, T.~Fujisawa, 
On semipositivity theorems, 
Math. Res. Lett. \textbf{26} (2019), no. 5, 1359--1382.

\bibitem[FnFs3]{fujino-fujisawa3}
O.~Fujino, T.~Fujisawa, 
Variation of mixed Hodge structure and its applications, 
preprint (2023). arXiv:2304.00672 [math.AG]

\bibitem[FnFsS]{fujino-fujisawa-saito}
O.~Fujino, T.~Fujisawa, M.~Saito, 
Some remarks on the semipositivity theorems, 
Publ. Res. Inst. Math. Sci. \textbf{50} (2014), no. 1, 85--112. 

\bibitem[Fs1]{Fujisawa} 
T.~Fujisawa, Mixed Hodge structures on log smooth degenerations, 
Tohoku Math. J. (2) \textbf{60} (2008), no. 1, 71--100. 

\bibitem[Fs2]{fujisawa2} 
T.~Fujisawa, 
A remark on semipositivity theorems, 
preprint (2017). arXiv:1710.01008 [math.AG]

\bibitem[Ft]{fujita} 
T.~Fujita, 
On K\"ahler fiber spaces over curves, 
J. Math. Soc. Japan \textbf{30} (1978), no. 4, 779--794.

\bibitem[H1]{hashizume-non} 
K.~Hashizume, 
On the non-vanishing conjecture and existence of log minimal models, 
Publ. Res. Inst. Math. Sci. \textbf{54} (2018), no. 1, 89--104.

\bibitem[H2]{hashizume}
K.~Hashizume, 
Log Iitaka conjecture for abundant log canonical fibrations, 
Proc. Japan Acad. Ser. A Math. Sci. \textbf{96} (2020), no. 10, 87--92.

\bibitem[I]{Illusie} 
L.~Illusie, {\em{Complexe cotangent et d\'eformations.~I}}, 
Lecture Notes in Math., vol. \textbf{239}. Springer-Verlag, 
Berlin--New York, 1971. 

\bibitem[Ks]{kashiwara} 
M.~Kashiwara, 
A study of variation of mixed Hodge structure, 
Publ. Res. Inst. Math. Sci. \textbf{22} (1986), no. 5, 991--1024. 

\bibitem[KtO]{katzoda}
N.~Katz, T.~Oda,
On the differentiation of De Rham cohomology classes with respect to parameters,
J. Math. Kyoto Univ. {\textbf{8}} (1968), 199--213.

\bibitem[Kw]{kawamata1}
Y.~Kawamata, 
Characterization of abelian varieties, 
Compositio Math. \textbf{43} (1981), no. 2, 253--276.

\bibitem[L]{lazarsfeld} 
R.~Lazarsfeld, {\em{Positivity in algebraic geometry. II. Positivity 
for vector bundles, and multiplier ideals}}, Ergebnisse der 
Mathematik und ihrer Grenzgebiete. 3. Folge. A Series of Modern Surveys in Mathematics [Results in Mathematics and Related 
Areas. 3rd Series. A Series of Modern Surveys in 
Mathematics], \textbf{49}. Springer-Verlag, Berlin, 2004.
 
\bibitem[Ma]{maehara} 
K.~Maehara, 
The weak $1$-positivity of direct 
image sheaves, 
J. Reine Angew. Math. \textbf{364} (1986), 112--129.

\bibitem[Mo]{mori}
S.~Mori, 
Classification of higher-dimensional varieties, 
{\em{Algebraic geometry, Bowdoin, 1985 (Brunswick, 
Maine, 1985)}}, 269--331, Proc. Sympos. Pure Math., \textbf{46}, Part 1, 
Amer. Math. Soc., Providence, RI, 1987.

\bibitem[N]{nakayama} 
N.~Nakayama, 
{\em{Zariski-decomposition and abundance}}, 
MSJ Memoirs, \textbf{14}. Mathematical Society of Japan, Tokyo, 2004. 

\bibitem[Pa]{park} 
S.~G.~Park, Logarithmic base change 
theorem and smooth descent of positivity of log canonical divisor, 
preprint (2022). arXiv:2210.02825 [math.AG]

\bibitem[PeS]{Peters-SteenbrinkMHS} 
C.~A.~M.~Peters, J.~H.~M.~Steenbrink, {\em{Mixed Hodge structures}}, 
Ergebnisse der Mathematik und ihrer Grenzgebiete. 3. Folge. A 
Series of Modern Surveys in Mathematics [Results in 
Mathematics and Related Areas. 3rd Series. A 
Series of Modern Surveys in Mathematics], \textbf{52}. Springer-Verlag, 
Berlin, 2008. 

\bibitem[Po]{popa}
M.~Popa, Conjectures on the Kodaira dimension, to appear in the
London Math. Soc. Lecture 
Note Ser. volume in honor of V. Shokurov's 70th birthday.

\bibitem[PoS1]{popa-schnell1} 
M.~Popa, C.~Schnell, 
Viehweg's hyperbolicity conjecture for 
families with maximal variation, 
Invent. Math. \textbf{208} (2017), no. 3, 677--713.

\bibitem[PoS2]{popa-schnell2} 
M.~Popa, C.~Schnell, 
On the behavior of the Kodaira dimension under smooth morphisms, 
Algebr. Geom. \textbf{10} (2023), no. 5, 607--619.

\bibitem[PoW]{popa-wu}
M.~Popa, L.~Wu, 
Weak positivity for Hodge modules, 
Math. Res. Lett. \textbf{23} (2016), no. 4, 1139--1155.

\bibitem[Sc]{schmid} 
W.~Schmid, 
Variation of Hodge structure: the singularities of the period mapping, 
Invent. Math. \textbf{22} (1973), 211--319.

\bibitem[St]{Steenbrink} 
J.~Steenbrink, 
Logarithmic embeddings of varieties with normal crossings and mixed 
Hodge structures, Math. Ann. \textbf{301} (1995), no. 1, 105--118.

\bibitem[StZ]{steenbrink-zucker}
J.~Steenbrink, S.~Zucker, 
Variation of mixed Hodge structure.~I, 
Invent. Math. \textbf{80} (1985), no. 3, 489--542.

\bibitem[U]{ueno} 
K.~Ueno, {\em{Classification theory of algebraic 
varieties and compact complex spaces}}, 
Notes written in collaboration with P.~Cherenack, 
Lecture Notes in Mathematics, Vol. \textbf{439}. Springer-Verlag, 
Berlin-New York, 1975.

\bibitem[VZ1]{viehweg-zuo1}
E.~Viehweg, K.~Zuo, 
On the isotriviality of families of projective manifolds over 
curves, J. Algebraic Geom. \textbf{10} (2001), no. 4, 781--799. 

\bibitem[VZ2]{viehweg-zuo2}
E.~Viehweg, K.~Zuo, 
Base spaces of non-isotrivial families 
of smooth minimal models, {\em{Complex geometry (G\"ottingen, 2000)}}, 
279--328, Springer, Berlin, 2002. 

\bibitem[Zuc]{Zucker}
S.~Zucker,
Remarks on a theorem of Fujita,
J. Math. Soc. Japan  \textbf{34} (1982), no.~1, 48--54.

\bibitem[Zuo]{zuo}
K.~Zuo, 
On the negativity of kernels of Kodaira--Spencer maps 
on Hodge bundles and applications, 
Kodaira's issue, Asian J. Math. \textbf{4} (2000), no. 1, 279--301. 

\end{thebibliography}
\end{document}